\documentclass[12pt]{article}

\usepackage{amsthm,amssymb,mathtools} 
\usepackage{color}
\usepackage{enumitem}
\usepackage{hyperref}
\usepackage{nicematrix}
\usepackage[subrefformat=parens]{subcaption}
\usepackage[scr=rsfs]{mathalpha}
\usepackage{empheq}
  
\usepackage[nameinlink,capitalise]{cleveref}
\crefname{figure}{Figure}{Figures}
\crefname{table}{Table}{Tables}

\usepackage{tikz-cd} 
\usepackage{tikz-3dplot}
\usetikzlibrary{patterns}
\usetikzlibrary{angles}

\hypersetup{
     colorlinks=true,
     linkcolor=black,
     filecolor=blue,
     citecolor = black,
     urlcolor=cyan,
     }

\makeatletter

\let\chige\c
\let\dotlessi\i

\let\caron\v

\@tempcnta\z@
\loop\ifnum\@tempcnta<26
\advance\@tempcnta\@ne
\expandafter\edef\csname \@alph\@tempcnta\endcsname{{\noexpand\boldsymbol \@alph\@tempcnta}}
\expandafter\edef\csname \@Alph\@tempcnta\endcsname{{\noexpand\boldsymbol \@Alph\@tempcnta}}
\expandafter\edef\csname b\@Alph\@tempcnta\endcsname{\noexpand\mathbb{\@Alph\@tempcnta}}
\expandafter\edef\csname \@Alph\@tempcnta C\endcsname{\noexpand\mathcal{\@Alph\@tempcnta}}
\expandafter\edef\csname f\@Alph\@tempcnta\endcsname{\noexpand\mathfrak{\@Alph\@tempcnta}}
\repeat

\@tempcnta\z@
\loop\ifnum\@tempcnta<10
\expandafter\edef\csname \@arabic\@tempcnta\endcsname{{\noexpand\boldsymbol \@arabic\@tempcnta}}
\advance\@tempcnta\@ne
\repeat

\def\Natural{\mbox{$\mathbb{N}$}}
\def\Real{\mbox{$\mathbb{R}$}}

\def\SymMat{\mbox{$\mathbb{S}$}}

\makeatother

\makeatletter

\DeclareMathOperator{\trace}{tr}
\DeclareMathOperator{\rank}{rank}
\DeclareMathOperator{\diag}{diag}

\newcommand{\trans}[1]{#1^{\operatorname{T}}}

\newcommand{\ip}[2]{{#1} \bullet {#2}}

\newcommand{\subto}{\mbox{s.t.}}

\makeatother
\makeatletter

\newtheorem{theorem}{Theorem}[section]
\newtheorem*{theorem*}{Theorem}
\newtheorem{prop}[theorem]{Proposition}
\newtheorem*{prop*}{Proposition}
\newtheorem{lemma}[theorem]{Lemma}
\newtheorem*{lemma*}{Lemma}

\newtheorem*{coro*}{Corollary}
\newtheorem{definition}[theorem]{Definition}
\newtheorem*{definition*}{Definition}
\newtheorem{assum}[theorem]{Assumption}
\newtheorem*{assum*}{Assumption}

\newtheorem*{rema*}{Remark}

\makeatother
\def\@themcountersep{}
\usepackage{fancybox,color}
\definecolor{lred}{rgb}{1,0.8,0.5}
\definecolor{lblue}{rgb}{0.8,0.8,1}
\definecolor{dred}{rgb}{0.6,0,0}
\definecolor{dblue}{rgb}{0,0,0.7}
\definecolor{violet}{rgb}{0.5804,0.0000,0.8275}
\definecolor{purple}{rgb}{0.2400,0.5700,0.2500}
\definecolor{TGreen}{rgb}{0,0.50,0.10}

\usepackage[mathlines]{lineno} 
\usepackage{amsmath}

\usepackage{etoolbox} 

\usepackage{easy-todo}

\newcommand*\linenomathpatch[1]{%
    \cspreto{#1}{\linenomath}%
    \cspreto{#1*}{\linenomath}%
    \csappto{end#1}{\endlinenomath}%
    \csappto{end#1*}{\endlinenomath}%
}
\newcommand*\linenomathpatchAMS[1]{%
    \cspreto{#1}{\linenomathAMS}%
    \cspreto{#1*}{\linenomathAMS}%
    \csappto{end#1}{\endlinenomath}%
    \csappto{end#1*}{\endlinenomath}%
}

\expandafter\ifx\linenomath\linenomathWithnumbers
\let\linenomathAMS\linenomathWithnumbers
\patchcmd\linenomathAMS{\advance\postdisplaypenalty\linenopenalty}{}{}{}
\else
\let\linenomathAMS\linenomathNonumbers
\fi

\linenomathpatch{equation}
\linenomathpatchAMS{gather}
\linenomathpatchAMS{multline}
\linenomathpatchAMS{align}
\linenomathpatchAMS{alignat}
\linenomathpatchAMS{flalign}

\usepackage[numbers, sort]{natbib}

\title{Tight Semidefinite Relaxations for Verifying Robustness of 
Neural Networks}

\makeatletter
\let\@fnsymbol\@arabic
\makeatother

\author{
\normalsize
    Godai Azuma\thanks{Department of Industrial and Systems Engineering,
    Aoyama Gakuin University, 5-10-1-O-410b Fuchinobe, Chuo-ku, Sagamihara-shi, Kanagawa 252-5258, Japan ({\tt azuma@ise.aoyama.ac.jp}).
    The research of Godai Azuma was supported by JSPS KAKENHI Grant Number JP24K20738.}
    \textsuperscript{,}\thinspace \thanks{Department of  Mathematical and Computing Science,
    Institute of Science Tokyo, 2-12-1-W8-29 Oh-Okayama, Meguro-ku, Tokyo 152-8550, Japan.
    ({\tt Makoto.Yamashita@comp.isct.ac.jp}).
    The research of Makoto Yamashita was partially supported by JSPS KAKENHI Grant Number 24K14836.}
\and
\normalsize
	Sunyoung Kim\thanks{Department of Mathematics, Ewha W. University, 52 Ewhayeodae-gil, Sudaemoon-gu,
	Seoul 03760, Korea  ({\tt skim@ewha.ac.kr}).
    This work was supported by NRF 2021-R1A2C1003810.}
\and
\normalsize
    Makoto Yamashita\footnotemark[2]
    }

\begin{document}
\maketitle


\begin{abstract} 
 \noindent
For verifying the safety of neural networks (NNs), Fazlyab et al.~(2019) introduced 
a semidefinite programming (SDP) 
approach called DeepSDP.  This formulation can be viewed as the dual of the SDP relaxation for a problem formulated as a quadratically constrained quadratic program (QCQP).
While SDP relaxations of QCQPs generally provide approximate solutions with some gaps,
this work focuses on tight SDP relaxations that provide exact solutions to the QCQP for single-layer NNs.
Specifically, we  analyze  tightness conditions in three cases:
(i) NNs with a single neuron, (ii) single-layer NNs with an ellipsoidal input set, and (iii) single-layer NNs with a rectangular input set.
For  NNs with a single neuron,  we propose a condition that ensures the SDP admits
a rank-1 solution to DeepSDP by transforming the QCQP
into an equivalent two-stage problem leads to a solution collinear with a predetermined vector.
For single-layer NNs with an ellipsoidal input set, the collinearity of solutions
is proved via the Karush-Kuhn-Tucker condition in the two-stage problem.
In case of single-layer NNs with a rectangular input set, we demonstrate that
 the tightness of DeepSDP can be reduced to the single-neuron NNs, case (i), if the weight matrix is a diagonal matrix.

\end{abstract}

\vspace{0.5cm}


\vspace{0.5cm}


\section{Introduction} \label{sec:introduction}

%

Neural networks (NNs) have gained significant attention
over the past two decades due to their theoretical foundations and practical applications in various fields, including
 image processing~\cite{egmont2002image}, sound recognition~\cite{deng2013new}, and natural language processing~\cite{goldberg2016primer,goldberg2017neural}.
Ensuring the robustness and safety of NNs against 
adversarial perturbations to test inputs
 is a key challenge in safe machine learning. 
To address this issue, we present a semidefinite programming (SDP)
 verifier for a neural network by showing  that the SDP relaxation 
provides exact solutions under suitable assumptions.

Consider an $L$-layer feed-forward NN~\cite{goldberg2017neural} described as a function $f: \Real^{n_0} \to \Real^{n_{L+1}}$
of form:
\begin{equation} \label{eq:neural_network}
	\left.
	\begin{aligned}
		\x^{k+1} &\coloneqq \phi\left(W^k \x^k + \b^k\right), \quad k = 0, \ldots, L - 1, \\
		f(\x^0)  &\coloneqq W^L \x^L + \b^L,
	\end{aligned}\quad\right\}
\end{equation}
where $\x^0 \in \Real^{n_0}$ is the input vector,
    $\x^k \in \Real^{n_k} \ (k=1,\dots,L)$ corresponds to the neurons in the $k$th layer.
    The matrix $W^k \in \Real^{n_{k+1} \times n_{k}}$
    and the vector $\b^k \in \Real^{n_{k+1}}$ are referred to as the weight matrix and bias vector at the $k$th layer, respectively.
The function $\phi(\x)$ is a vector-valued function 
where each element is an activation function.
In this paper, we are interested in the
rectified linear unit (ReLU) activation function,
that is $\phi(\x)_i = \max\{\alpha x_i, \beta x_i\}$ with $\alpha < \beta$ for every $i \in \Natural$.
The ReLU function can be emulated with three quadratic constraints:
\begin{align}
    (\phi(\x)_i - \alpha x_i)(\phi(\x)_i - \beta x_i) = 0, 
    \ \phi(\x)_i \ge \alpha x_i,
    \ \phi(\x)_i \ge \beta x_i. 
    \label{eq:QC}
\end{align}
As a result, the NN in \eqref{eq:neural_network}
can be formulated as a quadratically-constrained quadratic programs (QCQP) with an appropriate objective function.

It is well known that small input perturbations $\varepsilon$ in NNs can lead to significant changes in the output  $f(\x^0 + \varepsilon)$ \cite{moosavi2017universal,su2019one},
which makes NNs vulnerable to adversarial inputs that produce incorrect results.
This vulnerability undermines safety-critical applications including self-driving cars.
To model the vulnerability behavior of NN~\eqref{eq:neural_network}, 
optimization methods such as QCQPs have been proposed with
 expressing the ReLU function  as quadratic inequalities.
Specifically,
    mixed-integer linear programming programs have been incorporated
    into the algorithms in \cite{bastani2016measuring,dutta2018output,lomuscio2017approach,tjeng2017evaluating}
    and the duality of the convex relaxation has also been used in \cite{wong2018provable,dvijotham2018dual,salman2019convex}.

The SDP relaxation has been extensively studied as a powerful tool for generating approximate solutions to QCQPs
\cite{Gartner2014,luo2009sdp}.
To address the robustness of NNs,
Zhang~\cite{Zhang2020} proposed a tight semidefinite program (SDP) relaxation
by analyzing the \textit{collinearity} of solutions of the SDP relaxation 
and presenting a condition 
under which the optimal value of the SDP relaxation is equivalent to that of the QCQP formulated for
a single-layer NN.
Based on this result,
the approach in \cite{Zhang2020} was able to generate adversarial inputs 
with high accuracy.


Another approach that employs SDP relaxations to understand the vulnerability behavior of NNs is 
DeepSDP, which was recently proposed by
Fazlyab et al.~\cite{Fazlyab2019,Fazlyab2022}.
DeepSDP has been developed in the context of 
safety verification~\cite{huang2017safety,huang2023verification}.
To evaluate the safety of NNs,
safety verification determines whether the output set 
$f(\XC) := \{f(\x^0) \mid \x^0 \in \XC \}$, where 
$\XC \subset \Real^{n_0}$ is a given input set,
remains within a 
predefined safety specification set $S_y \subset \Real^{n_{L+1}}$.
Fazylab et al.~\cite{Fazlyab2019, Fazlyab2022}
employed the S-lemma to formulate DeepSDP as the following SDP problem:
\begin{equation} \label{eq:introduction_deepsdp}
    \begin{array}{rl}
        \min\limits_{P,Q,S} & g(P,Q,S) \\
        \subto           & M_\mathrm{in}(P) + M_\mathrm{mid}(Q) + M_\mathrm{out}(S) \succeq O, \\
                         & P \in \PC_\XC,\; Q \in \QC_\phi,\; S \in \SC,
    \end{array}
\end{equation}
where $g$ is a convex function, $\PC_\XC$, $\QC_\phi$, and $\SC$ are matrix sets representing the domain space of variables.
The set $\PC_\XC$ describes the information of the input set $\XC$,
    $\QC_\phi$ emulates the ReLU activation function $\phi$, and $\SC$ specifies a safety specification
    set $S_y \subseteq \Real^{n_{L+1}}$,
    where $\Real^{n_{L+1}}$ denotes the $n_{L+1}$-dimensional Euclidean space of column vectors $\y = \trans{\left[y_1,\ldots, y_{n_{L+1}}\right]}$.
    The three functions $M_\mathrm{in}$, $M_\mathrm{mid}$ and $M_\mathrm{out}$ lift up the three variable matrices $P, Q, S$ to an appropriate dimension,
    and $X \succeq O$ denotes that $X$ is positive semidefinite.    
Fazlyab et al.~\cite{Fazlyab2019,Fazlyab2022} showed that 
if 
$(P,Q,S)$ is a feasible solution \eqref{eq:introduction_deepsdp},
DeepSDP provides an safety specification set $S_y \subset \Real^{n_L + 1}$ that encompasses the output set $f(\XC)$
(see Theorem~\ref{thm:fazlyab_theorem2} in Section~\ref{ssec:formulation_deepsdp} for details).

DeepSDP is less accurate than other safety verification methods as shown in~\cite{newton2021neural},
since \eqref{eq:introduction_deepsdp} is not always tight 
for the corresponding QCQP.
In this paper, we investigate the tightness of DeepSDP \eqref{eq:introduction_deepsdp} as the SDP relaxation 
for the QCQP, using the collinearity.
It should be noted that the approach in \citet{Zhang2020}
is inapplicable to DeepSDP~\eqref{eq:neural_network},
since DeepSDP~\eqref{eq:neural_network}
includes the information of 
the input set $\XC$ and the safety verification set $S_y$.
Our approach is to show the collinearity utilizing
the second projection theorem~\cite[Theorem~9.8]{Beck2014} and the Karush-Kuhn-Tucker conditions.

The main contribution of this paper is to provide  tightness conditions for DeepSDP \eqref{eq:introduction_deepsdp} in
(i)   single-neuron NNs, i.e., $L = 1$ and $n_0 = n_1 = n_2 = 1$;
(ii)  single-layer NNs with ellipsoidal inputs $\XC$ and polytope safety specification set $S_y$; and
(iii) single-layer NNs with rectangular inputs $\XC$ and polytope safety specification set $S_y$.
%
Whether  DeepSDP can provide a robust or exact solution can be analytically determined by examining the conditions.
For (ii) and (iii),
we consider the cases where $\XC$ is an ellipsoid and a hyper-rectangle
as the neighborhood $\left\{ \x \in \Real^{n_0} \,\middle|\, \left\| \x - \hat{\x} \right\| \leq \varepsilon\right\}$
    around the true input $\hat{\x}$ is commonly used for the input set $\XC$ to evaluate  uncertainty,
    where $\|\cdot\|$ is the $\ell_1$-norm or  $\ell_2$-norm, and $\varepsilon \in \Real$ is a given value.
    The problem in~\cite{Zhang2020}, while  focusing on the tightness of the SDP relaxation in single-layer NNs,
     was not related to safety verification.  
In this paper, we present 
 tightness  conditions for the SDP relaxation applied to safety verification, particularly, in the estimation of the minimum safe
 specification set.

This paper is organized as follows.
In Section~\ref{sec:preliminary}, we describe SDP relaxations 
of QCQPs and the collinearlity-based tightness condition in \citet{Zhang2020}.
Section~\ref{sec:tightness} gives a concrete formulation of DeepSDP and discusses its tightness as the SDP relaxation.
Sections~\ref{sec:oneneuron} and \ref{sec:deepsdp_onelayer}  include the main results of this paper.
In Section~\ref{sec:oneneuron}, we provide a tightness condition  for  DeepSDP \eqref{eq:introduction_deepsdp} in case of a single-neuron NN.
In Section~\ref{sec:deepsdp_onelayer}, 
    we derive tightness conditions for DeepSDP \eqref{eq:introduction_deepsdp}  with respect to two different shapes of input sets.
We finally conclude in Section~\ref{sec:conclusion}.

\section{Preliminaries} \label{sec:preliminary}

\subsection{Notation}

The symbol $\Real_+^n$ denotes the set of nonnegative vectors in $\Real^n$.
Let $\0 \in \Real^n$ and $\1 \in \Real^n$ be the zero vector and the vector of all ones. 
We also let $\e^i \in \Real^n$ be the $i$th unit vector of an appropriate length,
    {\it i.e.,} 
    $\left(\e^i\right)_i = 1$ and $\left(\e^i\right)_k = 0$ for all $k \neq i$.
We denote by $\SymMat^n$ the set of the $n \times n$ symmetric matrices.
For any symmetric matrices $A, B \in \SymMat^n$,
    $\ip{A}{B}$ denote the Frobenius inner product of $A$ and $B$ defined as 
$
    \ip{A}{B} \coloneqq \trace(AB) = \sum_{i=1}^n \sum_{j=1}^n A_{ij} B_{ij}.
$
For a subset $S \subset \Omega$, $\1_{S}(\cdot)$ denotes the indicator function of $S$,
     {\it i.e.,}  $\1_{S}(x) = 1$ if $x \in S$; $\1_{S}(x) = 0$ otherwise.
A matrix $\textrm{diag}(\x)$ denotes a diagonal matrix whose diagonal elements are $\x$.

Let $N \coloneqq \sum_{k=1}^L n_k$ be the total number of neurons in the entire  network.
If $\XC \subseteq \Real^{n_0}$ is a given input set of NN \eqref{eq:neural_network},
the output of $f$ on $\XC$ is $\YC \coloneqq \left\{ f(\x^0) \in \Real^{n_{L+1}} \,\middle|\, \x^0 \in \XC \right\}$.
Moreover, for the input set $\XC$ in \eqref{eq:neural_network}, 
    the sets $\XC_0, \ldots, \XC_L$ are defined by
    $\XC_0 \coloneqq \XC$ and $\XC_{k+1} \coloneqq \left\{ \phi(W^k \x^k + \b^k) \,\middle|\, \x^k \in \XC_{k} \right\}$.
%

\subsection{Tight SDP relaxations of QCQPs and  strong duality} \label{ssec:sdp_relaxation}

We consider an inequality standard form QCQP: 
\begin{equation} \label{eq:general_qcqp}
    \begin{array}{rl}
        \min\limits_{\x} & \trans{\x}Q^0\x + 2\trans{(\q^0)}\x \\
        \subto           & \trans{\x}Q^k\x + 2\trans{(\q^k)}\x \leq b_k, \quad k = 1,\ldots,m, \\
                         & \x \in \Real^n.
    \end{array}
\end{equation}
As QCQPs are  generally nonconvex, 
  a well-known approach for obtaining a good approximate optimal value  is to use SDP relaxations.
The SDP relaxation of \eqref{eq:general_qcqp} can be expressed  as: 
\begin{equation} \label{eq:general_sdp_primal}
    \begin{array}{rl}
        \min\limits_{\x, X} & \ip{Q^0}{X} + 2\trans{(\q^0)}\x \\
        \subto
                            & \ip{Q^k}{X} + 2\trans{(\q^k)}\x \leq b_k, \quad k = 1,\ldots,m, \\
                            & \begin{bmatrix} 1 & \trans{\x} \\ \x & X \end{bmatrix} \succeq O,
    \end{array}
\end{equation}
where $X$ is a $n \times n$ symmetric matrix.
The optimal value $\theta$ of \eqref{eq:general_qcqp} is generally bounded by
    the optimal value $\zeta$ of \eqref{eq:general_sdp_primal} from below, {\it i.e.}, $\theta \geq \zeta$.
If $\theta = \zeta$, we say that the SDP relaxation~\eqref{eq:general_sdp_primal} is tight (or equivalently, exact).
It is well known that \eqref{eq:general_sdp_primal} is tight if and only if \eqref{eq:general_sdp_primal} has a rank-1 solution.
A great deal of studies has been conducted for the tight SDP relaxation as
    the exact optimal value and/or solution of \eqref{eq:general_qcqp} can be computed in polynomial time with the SDP relaxation
    ~\cite{Argue2020necessary,Azuma2021,Azuma2022,Burer2019,Kilinckarzan2021exactness,kim2003exact,Sojoudi2014exactness,Wang2021tightness}.

The dual problem of \eqref{eq:general_sdp_primal} is
\begin{equation} \label{eq:general_sdp_dual}
    \begin{array}{rl}
        \max\limits_{\boldsymbol{\xi}, \psi} & -\trans{\b}\boldsymbol{\xi} + \psi \\
        \subto
            & \begin{bmatrix} -\psi & \trans{(\q^0)} \\ \q^0 & Q^0 \end{bmatrix} 
                + \sum\limits_{k = 1}^m \xi_k \begin{bmatrix} 0 & \trans{(\q^k)} \\ \q^k & Q^k \end{bmatrix} \succeq O, \\
            & \boldsymbol{\xi} \geq \0.
    \end{array}
\end{equation}
The optimal values of SDP~\eqref{eq:general_sdp_primal} and its dual problem~\eqref{eq:general_sdp_dual} generally do not coincide.
%
The following lemma is used in
\cite{Azuma2022} for strong duality, and it is based on the  sufficient conditions in \cite[Corollary~4.3]{kim2021strong}.
\begin{lemma}{\upshape \cite[Lemma~3.3]{Azuma2022}} \label{lem:strong_duality_assumption}
    If the pair of \eqref{eq:general_sdp_primal} and \eqref{eq:general_sdp_dual} satisfies both conditions:
    \begin{enumerate}[label=(\roman*)]
        \item \label{assum:strong-duality-assumption-1}
            both \eqref{eq:general_sdp_primal} and \eqref{eq:general_sdp_dual} have optimal solutions; and
        \item \label{assum:strong-duality-assumption-2}
            the set of optimal solutions for \eqref{eq:general_sdp_dual} is bounded,
    \end{enumerate}
    then strong duality holds between \eqref{eq:general_sdp_primal} and \eqref{eq:general_sdp_dual},
    {\it i.e.}, they have optimal solutions and their optimal values are finite and equal.
\end{lemma}
By Lemma \ref{lem:strong_duality_assumption}, we see that if \eqref{eq:general_sdp_primal} is tight for \eqref{eq:general_sdp_dual}
    and strong duality holds between \eqref{eq:general_sdp_dual} and \eqref{eq:general_sdp_primal},
    then \eqref{eq:general_sdp_dual} is also tight for \eqref{eq:general_sdp_dual}.
In \cref{sec:oneneuron,sec:deepsdp_onelayer},
    we demonstrate that the dual of DeepSDP~\eqref{eq:introduction_deepsdp} is tight under certain assumptions.
    Therefore, verifying strong duality between  DeepSDP \eqref{eq:introduction_deepsdp} and its dual is crucial
    when estimating the safety specification set with DeepSDP~\eqref{eq:introduction_deepsdp}. 

\subsection{Existence of a rank-1 solution in rank-constrained SDP} \label{ssec:tight_rank_constrained_SDP}

This section describes a necessary and sufficient condition  in~\citet{Zhang2020}
    for ensuring that a given rank-constrained SDP has a rank-1 matrix solution. 
    This condition is related to the tightness of the SDP relaxation, as shown in Lemma~\ref{lem:rank_one_solution_collinear} below. 
Consider the SDP relaxation~\eqref{eq:general_sdp_primal} with an additional rank constraint $\rank(X) \leq p$ of the following form: 
\begin{equation} \label{eq:general_rcsdp}
    \begin{array}{rl}
        \min\limits_{\x, X} & \ip{Q^0}{X} + 2\trans{(\q^0)}\x \\
        \subto
                            & \ip{Q^k}{X} + 2\trans{(\q^k)}\x \leq b_k, \quad k = 1,\ldots,m, \\
                            & G \coloneqq \begin{bmatrix} 1 & \trans{\x} \\ \x  & X \end{bmatrix} \succeq O, \quad \rank(X) \leq p.
    \end{array}
\end{equation}
The problem~\eqref{eq:general_rcsdp} reduces into the SDP relaxation \eqref{eq:general_sdp_primal} if $p = n$.
Obviously, if \eqref{eq:general_rcsdp} has a rank-1 solution for any $p \geq 1$,
    then  \eqref{eq:general_sdp_primal} also admits a rank-1 solution.
Therefore, \eqref{eq:general_sdp_primal} is a tight relaxation of \eqref{eq:general_qcqp}.

The key idea in \cite{Zhang2020} for analyzing a rank-1 feasible solution of \eqref{eq:general_rcsdp} is to
    express  $G$ in terms of the Gram matrix.
Fix an arbitrary vector $\e \in \Real^p$ such that $\|\e\| = 1$.
New variables $\u^1, \ldots, \u^n \in \Real^p$ are introduced to substitute $\x$ and $X$ with 
$$    
    G = \begin{bNiceMatrix}[vlines={2},hlines={2},margin] \trans{\e}\e & \trans{\x} \\ \x & X \end{bNiceMatrix}
      = \begin{bNiceMatrix}[vlines={2},hlines={2},margin]
        \trans{\e}\e & \trans{\e}\u^1 & \Cdots & \trans{\e}\u^n \\
        \trans{\e}\u^1 & \trans{\left(\u^1\right)}\u^1 & \Cdots & \trans{\left(\u^1\right)}\u^n \\
        \Vdots & \Vdots & \Ddots & \Vdots \\
        \trans{\e}\u^n & \trans{\left(\u^n\right)}\u^1 & \cdots & \trans{\left(\u^n\right)}\u^n \\
    \end{bNiceMatrix}.
$$
The rank-1 condition $\rank(G) = 1$ holds if the vectors $\u^1, \ldots, \u^n$ and $\e$ are all collinear.
\begin{definition}\label{def:collinear}
The vectors $\u^1, \ldots, \u^n$ and $\e$ are all collinear
if $\left|\trans{\e}\u^i\right| = \left\|\u^i\right\|$ for all $i \in \{1,\ldots,n\}$.
\end{definition}
By substituting $G$  in  \eqref{eq:general_rcsdp}
    with  the above Gram-matrix representation, \eqref{eq:general_rcsdp}  can be reformulated as a nonconvex optimization problem 
    in variables $\u^1,\ldots,\u^n \in \Real^p$:
\begin{equation} \label{eq:general_nc_interpretation}
    \begin{array}{cl}
        \min\limits_{\u^1,\ldots,\u^n}
               & \sum\limits_{i=1}^n\sum\limits_{j=1}^n Q^0_{ij} \trans{\left(\u^i\right)}\u^j + 2 \sum\limits_{i=1}^n q^0_i \trans{\e}\u^i \\
        \subto & \sum\limits_{i=1}^n\sum\limits_{j=1}^n Q^k_{ij} \trans{\left(\u^i\right)}\u^j + 2 \sum\limits_{i=1}^n q^k_i \trans{\e}\u^i \leq b_k,
                    \quad k = 1,\ldots,m, \\
               & \u^1,\ldots,\u^n \in \Real^p.
    \end{array}
\end{equation}
The following lemma allows us to determine whether
\eqref{eq:general_rcsdp} is a tight relaxation for the original problem.

\begin{lemma} \label{lem:rank_one_solution_collinear}
    Fix $\e \in \Real^p$ with $\|\e\| = 1$.
    Then, the following two conditions are equivalent.
    \begin{enumerate}[label=(\roman*)]
        \item the problem~\eqref{eq:general_rcsdp} has a rank-1 matrix solution.
        \item the problem \eqref{eq:general_nc_interpretation} has an optimal solution $(\u^1)^*, \ldots, (\u^n)^*$ which are collinear with $\e$.
    \end{enumerate}
    When either of (i) and (ii) holds, \eqref{eq:general_rcsdp} is a tight relaxation for \eqref{eq:general_qcqp}.
    In addition, $\x^*$ and $(\x^*, X^*)$ are optimal solutions of \eqref{eq:general_qcqp} and \eqref{eq:general_rcsdp}, respectively,
        where $\x^* \coloneqq \trans{\begin{bmatrix} \trans{\e}(\u^1)^* & \cdots & \trans{\e}(\u^n)^* \end{bmatrix}}$ and $X^* \coloneqq \x^*\trans{\left(\x^*\right)}$.
 \end{lemma}
\begin{proof}
    The equivalence between (i) and (ii) follows from \cite[Theorem~A.2]{Zhang2020}.
    Since \eqref{eq:general_rcsdp} has a rank-1 solution,
        it is a tight relaxation for \eqref{eq:general_qcqp}, as mentioned in \cref{ssec:sdp_relaxation}.
\end{proof}

The direction $\e \in \Real^p$ can be fixed for the discussion here, as described in Appendix~A of \cite{Zhang2020}.
For instance, let us consider a case that,
    after solving \eqref{eq:general_nc_interpretation} with $\e = \bar{\e}$,
    we wish to have a rank-1 solution of \eqref{eq:general_nc_interpretation} with $\e = \hat{\e}$.
We can find an orthonormal matrix $U \in \Real^{p \times p}$ such that $\hat{\e} = U\bar{\e}$ by the Gram-Schmidt process. 
Let $\bar{\u}^1,\ldots,\bar{\u}^n$ be an optimal solution of \eqref{eq:general_nc_interpretation} with $\e = \bar{\e}$.
Then, for all pair $(i,j) \in \left\{1,\ldots,n\right\}^2$,
    we have $\trans{\left(\bar{\u}^i\right)}\bar{\u}^j = \trans{\left(U\bar{\u}^i\right)}U\bar{\u}^j$
    and $\trans{\bar{\e}}\bar{\u}^j = \trans{\left(U\bar{\e}\right)}U\bar{\u}^j = \trans{\hat{\e}}U\bar{\u}^j$.
By taking $\hat{\u}^i \coloneqq U\bar{\u}^i$ for all $j \in \left\{1,\ldots,n\right\}$,
    an optimal solution $(\hat{\u}^1,\ldots,\hat{\u}^n)$ of \eqref{eq:general_nc_interpretation} with $\e = \bar{\e}$ is obtained.
Therefore, 
    we may assume that $\e = \e^1$ in the proofs without loss of generality throughout the paper.

\section{Tight  SDPs for  NNs} \label{sec:tightness}

In this section, we first describe the safety specification set of DeepSDP~\eqref{eq:introduction_deepsdp}. 
The connection between the tightness of \eqref{eq:introduction_deepsdp} and its accuracy is presented in \cref{ssec:between_tightness_and_deepsdp}.
Throughout, 
 DeepSDP \eqref{eq:introduction_deepsdp} with
the standard ReLU activation function $\phi$ is discussed,
{\it i.e.}, $\phi(\x)_i = \max\{0, x_i\}$ holds. 

\subsection{Safety specification sets} \label{ssec:safety_set}

To verify the safety of an input set $\XC$, 
it is important to analyze     the subset relationship between a given set $S_y$ and the image of $\XC$ under $f$, given by
$$
    \YC \coloneqq f(\XC) = \left\{ f(\x^0) \,\middle|\, \x^0 \in \XC \right\}.
$$
Let $\XC$ be an input set that either contains  uncertainty or represents an adversarial attack that we wish to evaluate.
We define a safety specification set $S_y \subset \Real^{n_{L+1}}$ such that,
    for every $\x^0 \in \XC$, it represents the  range over which the output value $f(\x^0)$ can be correctly interpreted. 
If $\YC \subseteq S_y$,
    then for any $\x^0 \in \XC$, the value $f(\x^0)$ is unaffected by uncertainty or adversarial attacks.
Thus, the inputs $\XC$ for the NN is considered safe.
To represent  candidate  safety specification sets,
DeepSDP~\eqref{eq:introduction_deepsdp} uses 
    the matrix set $\SC \subseteq \SymMat^{1 + n_0 + n_{L+1}}$, 
    where $\SymMat^{1 + n_0 + n_{L+1}}$ is the space of  symmetric matrices of dimension $1 + n_0 + n_{L+1}$.

\subsection{Quadratic constraints  in DeepSDP \texorpdfstring{\eqref{eq:introduction_deepsdp}}{} } \label{ssec:QC_ReLU}

We describe 
the quadratic constraints that constitute DeepSDP~\eqref{eq:introduction_deepsdp}.
In \cref{ssec:input_set}, we define the matrix set $\PC_\XC$ used in \eqref{eq:introduction_deepsdp}.
Quadratic constraints to encode the ReLU functions $\phi$ are discussed in \cref{sssec:globalqc_relu}.
\cref{sssec:repeated_nonlinearities} formulates valid cuts which strengthen constraints in DeepSDP~\eqref{eq:introduction_deepsdp}.
(See \cite{Fazlyab2022} for details.)

\subsubsection{Quadratic representation of input sets} \label{ssec:input_set}
For the nonempty input set $\XC \subseteq \Real^{n_0}$ of the NN,
$\PC_\XC \subset \SymMat^{n_0+1}$ is defined as the set of all symmetric indefinite matrices $P \in \SymMat^{n_0+1}$ such that
\begin{equation} \label{eq:def_PX}
    \trans{\begin{bmatrix} 1 \\ \x^0 \end{bmatrix}} P \begin{bmatrix} 1 \\ \x^0 \end{bmatrix} \leq 0 \quad \text{for all $\x^0 \in \XC$}.
\end{equation}
Obviously,
\begin{equation} \label{eq:inclusion_Px}
    \XC \subseteq \bigcap_{P \in \PC_\XC} \left\{ \x^0 \in \Real^{n_0} \,\middle|\, \trans{\begin{bmatrix} 1 \\ \x \end{bmatrix}} P \begin{bmatrix} 1 \\ \x \end{bmatrix} \leq 0 \right\},
\end{equation}
and $\PC_\XC$ serves as an over-approximation to the input set $\XC$ using an infinite number of quadratic constraints.
In~\cite{Fazlyab2022}, the set $\PC_\XC$ for common input set $\XC$ was discussed.
In particular,
    for a hyper-rectangle $\XC \coloneqq \left\{\x \in \Real^{n_0}\,\middle|\, \underline{\x} \leq \x \leq \overline{\x}\right\}$,
    where $\underline{\x}\in \Real^{n_0}$ and $\overline{\x} \in \Real^{n_0}$ are the given lower and upper bounds,
    the over-approximation set $\PC_\XC$ can be described by
$$
    \PC_\XC = \left\{
        \begin{bmatrix}
            2\trans{\underline{\x}}\diag(\boldsymbol{\gamma})\overline{\x} & -\trans{(\underline{\x}+\overline{\x})}\diag(\boldsymbol{\gamma}) \\
            - \diag(\boldsymbol{\gamma}) (\underline{\x}+\overline{\x}) & 2\diag(\boldsymbol{\gamma})
        \end{bmatrix} \,\middle|\,
        \boldsymbol{\gamma} \geq \0
        \right\},
$$
and \eqref{eq:inclusion_Px} holds with equality.

\subsubsection{Global constraints} \label{sssec:globalqc_relu}

Let $\varphi: \Real \to \Real$ be the standard ReLU function $\varphi(x) = \max\{0, x\}$. 
We let $\phi(\z) \coloneqq \trans{\begin{bmatrix} \varphi(z_1) & \cdots &  \varphi(z_N) \end{bmatrix}}$ of $\z \in \Real^N$
    since the NN~\eqref{eq:neural_network} invokes $N$ ReLU activation functions $\varphi$ to output $f(\x^0)$.
To simplify the subsequent discussion,
     we define $\w^k \coloneqq W^k\x^k + \b^k \in \Real^{n_{k+1}}$ for every $k = 0, \ldots, L-1$ and $\w \coloneqq \trans{\begin{bmatrix} \trans{(\w^0)} & \cdots & \trans{(\w^{L-1})}\end{bmatrix}} \in \Real^{N}$.

As all  activation functions in the NN~\eqref{eq:neural_network} can be expressed in terms of $\phi(\w)$,
    we now discuss the constraints between $\w$ and $\phi(\w)$.
For all $i = 1,\ldots, N$, the constraint $\varphi(w_i) = \max\{0, w_i\}$ for expressing the standard ReLU activation function
    can be equivalently transformed into quadratic constraints 
\begin{equation} \label{eq:primal_global_quadratic_constraint_each}
    \varphi(w_i)\left[\varphi(w_i) - w_i\right] = 0, \quad
    \varphi(w_i) \geq w_i, \quad \varphi(w_i) \geq 0.
\end{equation}
%
Moreover, \eqref{eq:primal_global_quadratic_constraint_each} can be equivalently rewritten in the matrix form:
\begin{subequations} \label{eq:primal_global_quadratic_constraint}
    \begin{align}
        \trans{\begin{bmatrix} 1 \\ \w \\ \phi(\w) \end{bmatrix}}
        \begin{bmatrix} 0 & \trans{\0} & \trans{\0} \\ \0 & O & -\e_i\trans{\e_i} \\ \0 & -\e_i\trans{\e_i} & 2 \e_i\trans{\e_i} \end{bmatrix}
        \begin{bmatrix} 1 \\ \w \\ \phi(\w) \end{bmatrix} &= 0, \quad i = 1,\ldots,N, \label{eq:primal_global_quadratic_constraint_1} 
       \end{align}
       \begin{align}      
        \trans{\begin{bmatrix} 1 \\ \w \\ \phi(\w) \end{bmatrix}}
        \begin{bmatrix} 0 & \trans{\e_i} & -\trans{\e_i} \\ \e_i & O & O  \\ -\e_i & O & O \end{bmatrix}
        \begin{bmatrix} 1 \\ \w \\ \phi(\w) \end{bmatrix} &\leq 0, \quad i = 1,\ldots,N, \label{eq:primal_global_quadratic_constraint_2} \\
        \trans{\begin{bmatrix} 1 \\ \w \\ \phi(\w) \end{bmatrix}}
        \begin{bmatrix} 0 & \trans{\0} & -\trans{\e_i} \\ \0 & O & O \\ -\e_i & O & O \\  \end{bmatrix}
        \begin{bmatrix} 1 \\ \w \\ \phi(\w) \end{bmatrix} &\leq 0, \quad i = 1,\ldots,N. \label{eq:primal_global_quadratic_constraint_3}
    \end{align}
\end{subequations}
In \cite{Fazlyab2022}, the constraints~\eqref{eq:primal_global_quadratic_constraint} were employed
and 
 referred to as global quadratic constraints.

 We mention that Fazlyab et al.~\cite{Fazlyab2022} also proposed weaker inequalities than \eqref{eq:primal_global_quadratic_constraint}, 
 called local quadratic constraints. 
The local quadratic constraints are not dealt with in this work,
since their descriptions require more assumptions
on the input space $\XC$ and the activation function.

\subsubsection{Valid cuts} \label{sssec:repeated_nonlinearities}

Let $\w = \trans{\begin{bmatrix} \trans{(\w^0)} & \cdots & \trans{(\w^{L-1})}\end{bmatrix}} \in \Real^N$ be a column vector.
For the standard ReLU activate function $\varphi(x) = \max\{0, x\}$, the inequalities
\begin{equation} \label{eq:sloperestriction_relu_yy_each}
    \left[\varphi(w_j) - \varphi(w_i)\right] \left[\varphi(w_j) - \varphi(w_i) - \left(w_j - w_i\right)\right] \leq 0
\end{equation}
hold for all $i, j$ $(1 \le i < j \le N)$.
These valid cuts that couple between neurons were called as repeated nonlinearities in \cite{Fazlyab2022},
    since the same function $\varphi$ repeatedly appears in the NN for many neurons.
The inequality~\eqref{eq:sloperestriction_relu_yy_each} can be equivalently written as
\begin{equation} \label{eq:sloperestriction_relu_yy}
    \trans{\begin{bmatrix} 1 \\ \w \\ \phi(\w) \end{bmatrix}}
    \begin{bmatrix}
        0  & \trans{\0} & \trans{\0} \\
        \0 & O & -(\e_i - \e_j)\trans{(\e_i - \e_j)} \\
        \0 &-(\e_i - \e_j)\trans{(\e_i - \e_j)} & 2(\e_i - \e_j)\trans{(\e_i - \e_j)} \end{bmatrix}
    \begin{bmatrix} 1 \\ \w \\ \phi(\w) \end{bmatrix} \leq 0,
\end{equation}
for all $i, j$ $(1 \le i < j \le N)$. 

The computation of $f(\x^0)$ of the NN can be accelerated by computing $\varphi$ in parallel.
While increasing the computational efficiency  may be  important, it is secondary  
to ensuring the solution accuracy in this paper. Our primary focus  is whether the constraints can serve as valid cuts to improve the accuracy.
In \cref{sec:oneneuron,sec:deepsdp_onelayer},
    we prove that DeepSDP~\eqref{eq:introduction_deepsdp} is tight under certain conditions without these valid cuts~\eqref{eq:sloperestriction_relu_yy}.
This implies that the tightness can be maintained even when valid cuts are added to \eqref{eq:introduction_deepsdp}.

\subsection{Formulation of DeepSDP} \label{ssec:formulation_deepsdp}

In this subsection, we present the formulation of DeepSDP~\eqref{eq:introduction_deepsdp}.
%
We now consider a QCQP with the global quadratic constraints (see \cref{sssec:globalqc_relu})
    and the repeated nonlinearities (see  \cref{sssec:repeated_nonlinearities}):
\begin{equation} \label{eq:qcqp_for_deepsdp}
    \begin{array}{rl}
        \min & \trans{\begin{bmatrix} 1 \\ \x^0 \\ W^L\x^L + \b^L \end{bmatrix}} H \begin{bmatrix} 1 \\ \x^0 \\ W^L\x^L + \b^L \end{bmatrix} \\
        \subto
             & \trans{\begin{bmatrix} 1 \\ \x^0 \end{bmatrix}} P \begin{bmatrix} 1 \\ \x^0 \end{bmatrix} \leq 0, \; \forall P \in \PC_\XC, \\
             & \eqref{eq:primal_global_quadratic_constraint}, \; \eqref{eq:sloperestriction_relu_yy}, \;
                \w \coloneqq \overline{W} \begin{bmatrix} \x^0 \\ \vdots \\ \x^L \end{bmatrix} + \overline{\b} , \;
                \phi(\w) \coloneqq \overline{E} \begin{bmatrix} \x^0 \\ \vdots \\ \x^L \end{bmatrix}, \\
             & \x^k \in \Real^{n_k}, \; k = 0,\ldots,L,
    \end{array}
\end{equation}
where $H \in \SymMat^{1+n_0+n_{L+1}}$ is a given coefficient matrix for finding an extreme point of the output space $\Real^{n_{L+1}}$ and 
$$
    \overline{W} \coloneqq \left[\begin{array}{c|c} \begin{matrix} W^0 & & O \\ & \ddots & \\ O & & W^{L-1} \end{matrix} & \begin{matrix} O \\ \vdots \\ O \end{matrix} \end{array}\right],     \quad
    \overline{\b} \coloneqq \begin{bmatrix} \b^0 \\ \vdots \\ \b^{L-1} \end{bmatrix}, \quad 
    \overline{E} \coloneqq \left[ \begin{matrix} E^1 \\ \vdots \\ E^L \end{matrix} \right].
$$
Here, $E^i \in \Real^{n_i \times (n_0 + \cdots + n_L)}$ is a matrix such that $E^i \left[\x^0; \dots; \x^L\right] = \x^i$,
{\it i.e.,}  $E^i$ extracts $\x^i$ from $\left[\x^0; \dots; \x^L\right]$.

To simplify \eqref{eq:qcqp_for_deepsdp} by representing the variable vectors in the constraints as $[1; \x^0; \cdots; \x^L]$,
    we define three matrix functions:
\begin{align*}
    M_\mathrm{in}(P) &= \trans{\begin{bmatrix} 1 & \trans{\0} \\ \0 & E^0 \end{bmatrix}} P \begin{bmatrix} 1 & \trans{\0} \\ \0 & E^0 \end{bmatrix}, \,
    M_\mathrm{mid}(Q) = \trans{\begin{bmatrix} 1 & \trans{\0} \\ \overline{\b} & \overline{W} \\ \0 & \overline{E} \end{bmatrix}} Q \begin{bmatrix} 1 & \trans{\0} \\ \overline{\b} & \overline{W} \\ \0 & \overline{E} \end{bmatrix}, \\
    M_\mathrm{out}(S) &= \trans{\begin{bmatrix} 1 & \trans{\0} \\ \0 & E^0 \\ \b^L & W^LE^L \end{bmatrix}} S \begin{bmatrix} 1 & \trans{\0} \\ \0 & E^0 \\ \b^L & W^LE^L \end{bmatrix}.
\end{align*}
%
Then, the SDP relaxation of \eqref{eq:qcqp_for_deepsdp} is
\begin{equation} \label{eq:sdp_for_deepsdp}
    \begin{array}{rl}
        \min & \ip{M_\mathrm{out}(H)}{G} \\
        \subto
             & \ip{M_\mathrm{in}(P)}{G} \leq 0, \; \forall P \in \PC_\XC, \\
             & \ip{M_\mathrm{mid}\left(
                \begin{bmatrix} 0 & \trans{\0} & \trans{\0} \\ \0 & O & -\e_i\trans{\e_i} \\ \0 & -\e_i\trans{\e_i} & 2 \e_i\trans{\e_i} \end{bmatrix}
                \right)}{G} = 0, \quad i = 1,\ldots,N, \\
             & \ip{M_\mathrm{mid}\left(
                \begin{bmatrix} 0 & \trans{\e_i} & -\trans{\e_i} \\ \e_i & O & O  \\ -\e_i & O & O \end{bmatrix}
                \right)}{G} \leq 0, \;
               \ip{M_\mathrm{mid}\left(
                \begin{bmatrix} 0 & \trans{\0} & -\trans{\e_i} \\ \0 & O & O \\ -\e_i & O & O \end{bmatrix}
                \right)}{G} \leq 0, \quad i = 1,\ldots,N, \\
             & \ip{M_\mathrm{mid}\left(
                \begin{bmatrix}
                    0  & \trans{\0} & \trans{\0} \\
                    \0 & O & -(\e_i - \e_j)\trans{(\e_i - \e_j)} \\
                    \0 &-(\e_i - \e_j)\trans{(\e_i - \e_j)} & 2(\e_i - \e_j)\trans{(\e_i - \e_j)}
                \end{bmatrix}\right)}{G} \leq 0, \quad 1 \leq i < j \leq N, \\
             & G \coloneqq \begin{bmatrix} 1 & \trans{\x} \\ \x & X \end{bmatrix} \succeq O, \quad \x \in \Real^{n_0 + N}, \quad X \in \SymMat^{n_0 + N}.
    \end{array}
\end{equation}

DeepSDP~\eqref{eq:introduction_deepsdp} is formulated as the dual problem of \eqref{eq:sdp_for_deepsdp}.
Let $\boldsymbol{\lambda}$, $\boldsymbol{\nu}$, $\boldsymbol{\eta}$, and $\boldsymbol{\mu}$ be
    dual variables for
    \eqref{eq:primal_global_quadratic_constraint_1},
    \eqref{eq:primal_global_quadratic_constraint_2}, 
    \eqref{eq:primal_global_quadratic_constraint_3}, and
    \eqref{eq:sloperestriction_relu_yy}, respectively.
 The terms associated with $M_\mathrm{mid}$, among the terms in the Lagrangian function of \eqref{eq:sdp_for_deepsdp}, 
 can be written as follows:
$$
\mathcal{Q}_\phi \coloneqq \left\{
    \begin{bmatrix} 0 & Q_{12} & Q_{13} \\ \trans{Q_{12}} & Q_{22} & Q_{23} \\ \trans{Q_{13}} & \trans{Q_{23}} & Q_{33} \end{bmatrix} \,\middle|\,
    \boldsymbol{\lambda}, \boldsymbol{\nu}, \boldsymbol{\eta} \in \Real_+^{N},\;  \mu_{ij} \ge 0,\, 1 \le i < j \le N
    \right\} \subset \SymMat^{1+2N}
$$
where
\begin{equation} \label{eq:Qij_on_globalqc_relu}
    \left.
    \begin{gathered}
        Q_{12} \coloneqq \trans{\boldsymbol{\nu}}, \quad
        Q_{13} \coloneqq - \trans{\boldsymbol{\nu}} - \trans{\boldsymbol{\eta}}, \quad
        Q_{22} \coloneqq O, \quad
        Q_{23} \coloneqq -\left(\diag\left(\boldsymbol{\lambda}\right) + T\right), \\
        Q_{33} \coloneqq 2\left(\diag\left(\boldsymbol{\lambda}\right) + T\right), \quad
        T      \coloneqq \sum_{i = 1}^{N - 1} \sum_{j = i + 1}^{N} \mu_{ij} (\e_i - \e_j)\trans{(\e_i - \e_j)}.
    \end{gathered} \quad \right\}
\end{equation}
With $\XC$,  $\mathcal{Q}_\phi$, and $\SC$ defined in \cref{ssec:safety_set},
    DeepSDP~\eqref{eq:introduction_deepsdp} can be obtained.
 We  can also use the following theorem in \cite{Fazlyab2022} to determine the safety specification set $S_y$ of a given NN.
\begin{theorem}{\upshape \cite[Theorem~2]{Fazlyab2022}} \label{thm:fazlyab_theorem2}
    Let $(P, Q, S)$ be a feasible solution of \eqref{eq:introduction_deepsdp}.
    Then, for any $\x^0 \in \XC$, it holds that
    \begin{equation} \label{eq:constraints_safety_region}
        \trans{\begin{bmatrix}1 \\ \x^0 \\ f(\x^0) \end{bmatrix}} S \begin{bmatrix}1 \\ \x^0 \\ f(\x^0) \end{bmatrix} \geq 0.
    \end{equation}
\end{theorem}
\noindent
Theorem~\ref{thm:fazlyab_theorem2} implies that
    once we find an optimal solution $(P^*, Q^*, S^*)$ of \eqref{eq:introduction_deepsdp},
    the output set $\YC$ is a subset of the safety specification set
$$
    S_y \coloneqq \left\{ \y \in \Real^{n_{L+1}} \,\middle|\, 
        \trans{\begin{bmatrix} 1 \\ \x^0 \\ \y \end{bmatrix}} S^* \begin{bmatrix}1 \\ \x^0 \\ \y \end{bmatrix} \geq 0
        \ \text{for all $\x^0 \in \XC$}
        \right\}.
$$

\subsection{DeepSDP \texorpdfstring{\eqref{eq:introduction_deepsdp}}{}  and tight SDP relaxations } \label{ssec:between_tightness_and_deepsdp}

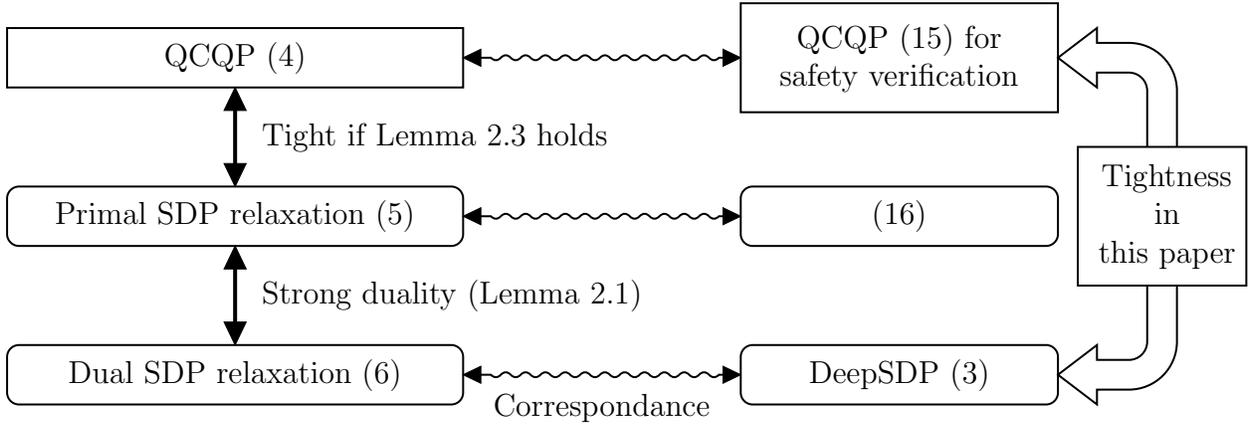
\begin{figure}[t]
    \tikzset{every picture/.style={line width=0.75pt}}
    \begin{tikzpicture}[x=0.75pt,y=0.75pt,yscale=-1,xscale=1]
        \draw  [fill={rgb, 255:red, 255; green, 255; blue, 255 }  ,fill opacity=1 ] (10,20) -- (240,20) -- (240,50) -- (10,50) -- cycle ;
        \draw  [fill={rgb, 255:red, 255; green, 255; blue, 255 }  ,fill opacity=1 ] (10,106) .. controls (10,102.69) and (12.69,100) .. (16,100) -- (234,100) .. controls (237.31,100) and (240,102.69) .. (240,106) -- (240,124) .. controls (240,127.31) and (237.31,130) .. (234,130) -- (16,130) .. controls (12.69,130) and (10,127.31) .. (10,124) -- cycle ;
        \draw  [fill={rgb, 255:red, 255; green, 255; blue, 255 }  ,fill opacity=1 ] (10,186) .. controls (10,182.69) and (12.69,180) .. (16,180) -- (234,180) .. controls (237.31,180) and (240,182.69) .. (240,186) -- (240,204) .. controls (240,207.31) and (237.31,210) .. (234,210) -- (16,210) .. controls (12.69,210) and (10,207.31) .. (10,204) -- cycle ;
        \draw [line width=1.5]    (125,54) -- (125,96) ;
        \draw [shift={(125,100)}, rotate = 270] [fill={rgb, 255:red, 0; green, 0; blue, 0 }  ][line width=0.08]  [draw opacity=0] (11.61,-5.58) -- (0,0) -- (11.61,5.58) -- cycle    ;
        \draw [shift={(125,50)}, rotate = 90] [fill={rgb, 255:red, 0; green, 0; blue, 0 }  ][line width=0.08]  [draw opacity=0] (11.61,-5.58) -- (0,0) -- (11.61,5.58) -- cycle    ;
        \draw [line width=1.5]    (125,134) -- (125,176) ;
        \draw [shift={(125,180)}, rotate = 270] [fill={rgb, 255:red, 0; green, 0; blue, 0 }  ][line width=0.08]  [draw opacity=0] (11.61,-5.58) -- (0,0) -- (11.61,5.58) -- cycle    ;
        \draw [shift={(125,130)}, rotate = 90] [fill={rgb, 255:red, 0; green, 0; blue, 0 }  ][line width=0.08]  [draw opacity=0] (11.61,-5.58) -- (0,0) -- (11.61,5.58) -- cycle    ;
        \draw  [fill={rgb, 255:red, 255; green, 255; blue, 255 }  ,fill opacity=1 ] (380,7.5) -- (540,7.5) -- (540,62.5) -- (380,62.5) -- cycle ;
        \draw  [fill={rgb, 255:red, 255; green, 255; blue, 255 }  ,fill opacity=1 ] (380,106) .. controls (380,102.69) and (382.69,100) .. (386,100) -- (534,100) .. controls (537.31,100) and (540,102.69) .. (540,106) -- (540,124) .. controls (540,127.31) and (537.31,130) .. (534,130) -- (386,130) .. controls (382.69,130) and (380,127.31) .. (380,124) -- cycle ;
        \draw  [fill={rgb, 255:red, 255; green, 255; blue, 255 }  ,fill opacity=1 ] (380,186) .. controls (380,182.69) and (382.69,180) .. (386,180) -- (534,180) .. controls (537.31,180) and (540,182.69) .. (540,186) -- (540,204) .. controls (540,207.31) and (537.31,210) .. (534,210) -- (386,210) .. controls (382.69,210) and (380,207.31) .. (380,204) -- cycle ;
        \draw    (377,195) -- (369,195) .. controls (367.33,196.67) and (365.67,196.67) .. (364,195) .. controls (362.33,193.33) and (360.67,193.33) .. (359,195) .. controls (357.33,196.67) and (355.67,196.67) .. (354,195) .. controls (352.33,193.33) and (350.67,193.33) .. (349,195) .. controls (347.33,196.67) and (345.67,196.67) .. (344,195) .. controls (342.33,193.33) and (340.67,193.33) .. (339,195) .. controls (337.33,196.67) and (335.67,196.67) .. (334,195) .. controls (332.33,193.33) and (330.67,193.33) .. (329,195) .. controls (327.33,196.67) and (325.67,196.67) .. (324,195) .. controls (322.33,193.33) and (320.67,193.33) .. (319,195) .. controls (317.33,196.67) and (315.67,196.67) .. (314,195) .. controls (312.33,193.33) and (310.67,193.33) .. (309,195) .. controls (307.33,196.67) and (305.67,196.67) .. (304,195) .. controls (302.33,193.33) and (300.67,193.33) .. (299,195) .. controls (297.33,196.67) and (295.67,196.67) .. (294,195) .. controls (292.33,193.33) and (290.67,193.33) .. (289,195) .. controls (287.33,196.67) and (285.67,196.67) .. (284,195) .. controls (282.33,193.33) and (280.67,193.33) .. (279,195) .. controls (277.33,196.67) and (275.67,196.67) .. (274,195) .. controls (272.33,193.33) and (270.67,193.33) .. (269,195) .. controls (267.33,196.67) and (265.67,196.67) .. (264,195) .. controls (262.33,193.33) and (260.67,193.33) .. (259,195) .. controls (257.33,196.67) and (255.67,196.67) .. (254,195) -- (251,195) -- (243,195) ;
        \draw [shift={(240,195)}, rotate = 360] [fill={rgb, 255:red, 0; green, 0; blue, 0 }  ][line width=0.08]  [draw opacity=0] (8.93,-4.29) -- (0,0) -- (8.93,4.29) -- cycle    ;
        \draw [shift={(380,195)}, rotate = 180] [fill={rgb, 255:red, 0; green, 0; blue, 0 }  ][line width=0.08]  [draw opacity=0] (8.93,-4.29) -- (0,0) -- (8.93,4.29) -- cycle    ;
        \draw    (377,115) -- (369,115) .. controls (367.33,116.67) and (365.67,116.67) .. (364,115) .. controls (362.33,113.33) and (360.67,113.33) .. (359,115) .. controls (357.33,116.67) and (355.67,116.67) .. (354,115) .. controls (352.33,113.33) and (350.67,113.33) .. (349,115) .. controls (347.33,116.67) and (345.67,116.67) .. (344,115) .. controls (342.33,113.33) and (340.67,113.33) .. (339,115) .. controls (337.33,116.67) and (335.67,116.67) .. (334,115) .. controls (332.33,113.33) and (330.67,113.33) .. (329,115) .. controls (327.33,116.67) and (325.67,116.67) .. (324,115) .. controls (322.33,113.33) and (320.67,113.33) .. (319,115) .. controls (317.33,116.67) and (315.67,116.67) .. (314,115) .. controls (312.33,113.33) and (310.67,113.33) .. (309,115) .. controls (307.33,116.67) and (305.67,116.67) .. (304,115) .. controls (302.33,113.33) and (300.67,113.33) .. (299,115) .. controls (297.33,116.67) and (295.67,116.67) .. (294,115) .. controls (292.33,113.33) and (290.67,113.33) .. (289,115) .. controls (287.33,116.67) and (285.67,116.67) .. (284,115) .. controls (282.33,113.33) and (280.67,113.33) .. (279,115) .. controls (277.33,116.67) and (275.67,116.67) .. (274,115) .. controls (272.33,113.33) and (270.67,113.33) .. (269,115) .. controls (267.33,116.67) and (265.67,116.67) .. (264,115) .. controls (262.33,113.33) and (260.67,113.33) .. (259,115) .. controls (257.33,116.67) and (255.67,116.67) .. (254,115) -- (251,115) -- (243,115) ;
        \draw [shift={(240,115)}, rotate = 360] [fill={rgb, 255:red, 0; green, 0; blue, 0 }  ][line width=0.08]  [draw opacity=0] (8.93,-4.29) -- (0,0) -- (8.93,4.29) -- cycle    ;
        \draw [shift={(380,115)}, rotate = 180] [fill={rgb, 255:red, 0; green, 0; blue, 0 }  ][line width=0.08]  [draw opacity=0] (8.93,-4.29) -- (0,0) -- (8.93,4.29) -- cycle    ;
        \draw    (377,35) -- (369,35) .. controls (367.33,36.67) and (365.67,36.67) .. (364,35) .. controls (362.33,33.33) and (360.67,33.33) .. (359,35) .. controls (357.33,36.67) and (355.67,36.67) .. (354,35) .. controls (352.33,33.33) and (350.67,33.33) .. (349,35) .. controls (347.33,36.67) and (345.67,36.67) .. (344,35) .. controls (342.33,33.33) and (340.67,33.33) .. (339,35) .. controls (337.33,36.67) and (335.67,36.67) .. (334,35) .. controls (332.33,33.33) and (330.67,33.33) .. (329,35) .. controls (327.33,36.67) and (325.67,36.67) .. (324,35) .. controls (322.33,33.33) and (320.67,33.33) .. (319,35) .. controls (317.33,36.67) and (315.67,36.67) .. (314,35) .. controls (312.33,33.33) and (310.67,33.33) .. (309,35) .. controls (307.33,36.67) and (305.67,36.67) .. (304,35) .. controls (302.33,33.33) and (300.67,33.33) .. (299,35) .. controls (297.33,36.67) and (295.67,36.67) .. (294,35) .. controls (292.33,33.33) and (290.67,33.33) .. (289,35) .. controls (287.33,36.67) and (285.67,36.67) .. (284,35) .. controls (282.33,33.33) and (280.67,33.33) .. (279,35) .. controls (277.33,36.67) and (275.67,36.67) .. (274,35) .. controls (272.33,33.33) and (270.67,33.33) .. (269,35) .. controls (267.33,36.67) and (265.67,36.67) .. (264,35) .. controls (262.33,33.33) and (260.67,33.33) .. (259,35) .. controls (257.33,36.67) and (255.67,36.67) .. (254,35) -- (251,35) -- (243,35) ;
        \draw [shift={(240,35)}, rotate = 360] [fill={rgb, 255:red, 0; green, 0; blue, 0 }  ][line width=0.08]  [draw opacity=0] (8.93,-4.29) -- (0,0) -- (8.93,4.29) -- cycle    ;
        \draw [shift={(380,35)}, rotate = 180] [fill={rgb, 255:red, 0; green, 0; blue, 0 }  ][line width=0.08]  [draw opacity=0] (8.93,-4.29) -- (0,0) -- (8.93,4.29) -- cycle    ;
        \draw   (585,150) -- (585,178.3) .. controls (585,183.38) and (580.88,187.5) .. (575.8,187.5) -- (559.67,187.5) -- (559.67,180) -- (540,195) -- (559.67,210) -- (559.67,202.5) -- (575.8,202.5) .. controls (589.17,202.5) and (600,191.67) .. (600,178.3) -- (600,150) -- (600,80) -- (600,51.7) .. controls (600,38.33) and (589.17,27.5) .. (575.8,27.5) -- (559.67,27.5) -- (559.67,20) -- (540,35) -- (559.67,50) -- (559.67,42.5) -- (575.8,42.5) .. controls (580.88,42.5) and (585,46.62) .. (585,51.7) -- (585,80) -- cycle ;

        \draw (137,155) node [anchor=west] [inner sep=0.75pt]   [align=left] {Strong duality (Lemma~\ref{lem:strong_duality_assumption})};
        \draw (137,75) node [anchor=west] [inner sep=0.75pt]   [align=left] {Tight if Lemma~\ref{lem:rank_one_solution_collinear} holds};
        \draw (125,115) node  [font=\normalsize] [align=left] {Primal SDP relaxation \eqref{eq:general_sdp_primal}};
        \draw (125,35) node  [font=\normalsize] [align=left] { QCQP \eqref{eq:general_qcqp}};
        \draw (125,195) node  [font=\normalsize] [align=left] {Dual SDP relaxation \eqref{eq:general_sdp_dual}};
        \draw (460,115) node  [font=\normalsize] [align=left] {\eqref{eq:sdp_for_deepsdp}};
        \draw (460,35) node  [font=\normalsize] [align=left] {\begin{minipage}[lt]{91.55pt}\setlength\topsep{0pt}
        \begin{center} QCQP \eqref{eq:qcqp_for_deepsdp} for\\safety verification \end{center}\end{minipage}};
        \draw (460,195) node  [font=\normalsize] [align=left] {DeepSDP \eqref{eq:introduction_deepsdp}};
        \draw (310,203) node [anchor=north] [inner sep=0.75pt]   [align=left] {Correspondance};
        \draw [fill={rgb, 255:red, 255; green, 255; blue, 255 }  ,fill opacity=1 ] (550, 80) -- (635, 80) -- (635, 150) -- (550, 150) -- cycle ;
        \draw (595,115) node [align=left] {\begin{minipage}[lt]{75.5pt}\setlength\topsep{0pt} \begin{center} Tightness \\ in \\this paper \end{center} \end{minipage}};
    \end{tikzpicture}

    \caption{Correspondence between QCQPs and DeepSDP  for the tightness.  The solid lines display the equivalent
     optimal value of the problems.  The wavy lines illustrate the corresponding relationship between the two problems.}
    \label{fig:correspondence_tightness}
\end{figure}

The tight SDP relaxation, discussed in the previous subsection,
    plays a crucial role for the accuracy of a solution to DeepSDP \eqref{eq:introduction_deepsdp}.
If the objective function $h$ of \eqref{eq:qcqp_for_deepsdp} is linear or quadratic,
    the problem \eqref{eq:qcqp_for_deepsdp} can be regarded as a QCQP~\eqref{eq:general_qcqp}.
In fact, all the constraints of \eqref{eq:qcqp_for_deepsdp} are at most degree $2$,
    and each  can be represented as
$$
    \trans{\begin{bmatrix} 1 \\ \x \end{bmatrix}}
    \begin{bmatrix} \tilde{b} & \trans{\tilde{\q}} \\ \tilde{\q} & \tilde{Q} \end{bmatrix} \begin{bmatrix} 1 \\ \x \end{bmatrix} \leq 0
$$
with $\x \coloneqq \left[ \x^0; \cdots ; \x^L \right] \in \Real^{n_0 + N}$,
    appropriate coefficients $\tilde{b} \in \Real$, $\tilde{\q} \in \Real^{n_0 + N}$, and $\tilde{Q} \in \SymMat^{n_0 + N}$. 
As DeepSDP~\eqref{eq:introduction_deepsdp} is the dual problem of the SDP relaxation~\eqref{eq:sdp_for_deepsdp} of QCQP \eqref{eq:qcqp_for_deepsdp},
    the relationship between \eqref{eq:qcqp_for_deepsdp} and \eqref{eq:introduction_deepsdp} is
    analogous to 
    the relationship between the standard QCQP \eqref{eq:general_qcqp} and its dual SDP relaxation \eqref{eq:general_sdp_dual}.
\cref{fig:correspondence_tightness} shows 
    the correspondence between \eqref{eq:qcqp_for_deepsdp} and \eqref{eq:introduction_deepsdp}.
    In the subsequent discussion,
    we refer to \eqref{eq:sdp_for_deepsdp} as the primal SDP relaxation of \eqref{eq:qcqp_for_deepsdp},
    and distinguish it from the dual SDP relaxation, DeepSDP~\eqref{eq:introduction_deepsdp}.
When strong duality holds between these relaxations \eqref{eq:introduction_deepsdp} and \eqref{eq:sdp_for_deepsdp},
    \eqref{eq:introduction_deepsdp} is a tight relaxation of \eqref{eq:qcqp_for_deepsdp}
    if and only if \eqref{eq:sdp_for_deepsdp} is also a tight relaxation of it, as shown in \cref{fig:correspondence_tightness}.
We focus on the tightness of DeepSDP~\eqref{eq:introduction_deepsdp} for \eqref{eq:qcqp_for_deepsdp}.

Since the cardinality of $\PC_\XC$ can be infinite, 
selecting the appropriate constraints is crucial.
In subsequent discussion, the structure of $\PC_\XC$ can be simplified by restricting the input set $\XC$
to an ellipsoid or a rectangle.
With a given center $\hat{\x}$ and a radius $\varepsilon$,
a ellipsoidal input set can be represented by $\left\{\x \in \Real^{n_0}\,\middle|\,\left\|\x-\hat{\x}\right\|_2\leq\varepsilon\right\}$;
and a rectangular input set  by $\left\{\x \in \Real^{n_0}\,\middle|\,\left\|\x-\hat{\x}\right\|_1\leq\varepsilon\right\}$.

Other factors can also  affect the accuracy of DeepSDP \eqref{eq:introduction_deepsdp}.
For example, if $\SC$ does not  contain the necessary matrices to represent the true safety specification set, the recovered set will not be minimal.
We do not address other factors, such as the selection of $\SC$, in this paper.

\section{Tightness of DeepSDP \texorpdfstring{\eqref{eq:introduction_deepsdp}}{} for a single-neuron ReLU network} \label{sec:oneneuron}

We first analyze the tightness of DeepSDP \eqref{eq:introduction_deepsdp} with a single-neuron case of the following:
$$
    f(x) = \varphi\left(x + b^0\right), \quad x \in \XC \subseteq \Real,
$$
which is equivalent to 
the NN~\eqref{eq:neural_network} with $L = 1$, $n_0 = n_1 = 1$, $W^0 = W^1 = 1$ and $\b^1 = 0$. 
The output $y \coloneqq f(x^0)$ is $x^1$ in \eqref{eq:neural_network}.
In this case, the smallest safety specification set is clearly $\XC \cap \Real_+$ as only one ReLU activation function is applied in the NN.
    The following assumptions hold throughout this section.
    \begin{assum} \label{asm:X_box}
        The input set $\XC$ is a closed interval,
        i.e., 
        $\XC = \{ x \,|\, \underline{x} \leq x \leq \bar{x}\}$ 
        for some $\underline{x}$ and $\bar{x}$.
    \end{assum}
    \begin{assum} \label{asm:Sy_poly}
        The candidate of the safety specification set $S_y$ on the output space is a polytope.
    \end{assum}
Under these assumptions,
the safety set $S_y$ can be represented by 
a closed interval $[\underline{d}, \overline{d}]$,
which can be further rewritten as 
$\left\{ y \in \Real \,\middle|\, y - \underline{d} \geq 0 \right\} \cap
\left\{ y \in \Real \,\middle|\, -y - (-\overline{d}) \geq 0 \right\}$.
For $c \in \{-1, +1\}$ and $d \in \Real$, to determine whether $f(x)$ lies within $\left\{ y \in \Real \,\middle|\, cy - d \geq 0 \right\}$ for all $x \in \XC$,
we check if the following equation holds:
$$
        \trans{\begin{bmatrix} 1 \\ x \\ f(x) \end{bmatrix}}
        \begin{bmatrix}
            -2d & 0 & c \\ 0 & 0 & 0 \\ c & 0 & 0
        \end{bmatrix}
        \begin{bmatrix} 1 \\ x \\ f(x) \end{bmatrix} \geq 0.
$$
By Assumptions 
\ref{asm:X_box} and  \ref{asm:Sy_poly},
 the matrix sets $\PC_\XC$, $\QC_\phi$, and $\SC$ can be written as
\begin{gather*}
    \PC_\XC = \left\{
        \begin{bmatrix}
            2 \underline{x} \bar{x} \gamma & - \left(\underline{x} + \bar{x} \right)\gamma \\
            - \left(\underline{x} + \bar{x} \right)\gamma & 2\gamma
        \end{bmatrix} \,\middle|\, \gamma \geq 0
        \right\}, \\
    \QC_\phi = \left\{
        \begin{bmatrix}
            2b^0\nu & \nu & - \nu - \eta - b^0\lambda \\ \nu & 0 & - \lambda \\ - \nu - \eta -b^0 \lambda & - \lambda & 2\lambda
        \end{bmatrix} \,\middle|\, \lambda, \nu, \eta \in \Real_+ \right\}, \\
    \SC = \left\{
        \begin{bmatrix}
            -2d & 0 & c \\ 0 & 0 & 0 \\ c & 0 & 0
        \end{bmatrix}
        \,\middle|\, d \in \Real \right\},
\end{gather*}
with $\gamma, \lambda, \nu, \eta \in \Real_+$ and $d \in \Real$.
Since there is only one neuron, 
    no constraints exist for the relationship between two neurons, and
    the repeated nonlinearities do not appear, {\it i.e.}, $T = 0$ in \eqref{eq:introduction_deepsdp}.
Consequently,
    DeepSDP~\eqref{eq:introduction_deepsdp} for a single-neuron NN 
    can be described as
    \begin{equation} \label{eq:deepsdp_oneneuron}
    \begin{array}{rl}
        \max & 2d \\
        \subto
             & \begin{bmatrix}
                   2 \underline{x} \bar{x} \gamma & - \left(\underline{x} + \bar{x} \right)\gamma & 0 \\
                   - \left(\underline{x} + \bar{x} \right)\gamma & 2\gamma & 0 \\
                   0 & 0 & 0
               \end{bmatrix} +
               \begin{bmatrix}
                   2b^0\nu & \nu & - \nu - \eta - b^0\lambda \\ \nu & 0 & - \lambda \\ - \nu - \eta - b^0\lambda & - \lambda & 2\lambda
               \end{bmatrix}  \\
             & \qquad
               + \begin{bmatrix}
                   -2d & 0 & c \\ 0 & 0 & 0 \\ c & 0 & 0
               \end{bmatrix} \succeq O, \ \  \gamma, \lambda, \nu, \eta \in \Real_+, \quad d \in \Real.
    \end{array}
\end{equation}
%
Here, we maximize $d$ 
under a fixed $c$ as DeepSDP \eqref{eq:introduction_deepsdp} is to find the minimum interval 
(the minimal safety specification set).

By  demonstrating that \eqref{eq:deepsdp_oneneuron} is a tight SDP relaxation in the subsequent discussion,
    we can obtain an optimal solution $d^*$, which is also an optimal solution of the QCQP formulation~\eqref{eq:qcqp_for_deepsdp}. 

\subsection{Two-stage formulation}

We derive a nonconvex optimization problem of the form~\eqref{eq:general_nc_interpretation} in this section.
  
The dual of the SDP~\eqref{eq:deepsdp_oneneuron} is:  
\begin{equation} \label{eq:primal_deepsdp_oneneuron}
    \begin{array}{rl}
        \min\limits_{x^0, x^1, X} & \ip{\begin{bmatrix} 0 & 0 & c \\ 0 & 0 & 0 \\ c & 0 & 0 \end{bmatrix}}{G} \\
        \subto
            & \ip{\begin{bmatrix}
                   2 \underline{x} \bar{x} & - \left(\underline{x} + \bar{x} \right) & 0 \\
                   - \left(\underline{x} + \bar{x} \right) & 2 & 0 \\
                   0 & 0 & 0
               \end{bmatrix}}{G} \leq 0, \;
                \ip{\begin{bmatrix} -2 & 0 & 0 \\ 0 & 0 & 0 \\ 0 & 0 & 0 \end{bmatrix}}{G} \leq -2, \\
            & \ip{\begin{bmatrix} 2b^0 & 1 & -1 \\ 1 & 0 & 0 \\ -1 & 0 & 0 \end{bmatrix}}{G} \leq 0, \;
                \ip{\begin{bmatrix} 0 & 0 & -1 \\ 0 & 0 & 0 \\ -1 & 0 & 0 \end{bmatrix}}{G} \leq 0, \;
                \ip{\begin{bmatrix} 0 & 0 & -b^0 \\ 0 & 0 & -1 \\ -b^0 & -1 & 2 \end{bmatrix}}{G} \leq 0, \\
            & G \coloneqq \begin{bNiceMatrix}[vlines=2,hlines=2,margin]
                1 & x^0 & x^1 \\
                x^0 & \Block{2-2}{X} & \\
                x^1 & & 
            \end{bNiceMatrix} \in \SymMat_+^3, \quad x^0, x^1 \in \Real, \quad X \in \SymMat_+^2.
    \end{array}
\end{equation}
Note that \eqref{eq:primal_deepsdp_oneneuron} is
    a rank-constrained SDP~\eqref{eq:general_rcsdp} with $p = 3$,
    and corresponds to an SDP relaxation of the following QCQP:
\begin{equation*} 
    \begin{array}{rl}
        \min\limits_{x^0, x^1} & c x^1 \\
        \subto
            & \left(\bar{x} - x^0\right)\left(x^0 - \underline{x}\right) \geq 0, \; -2 \leq -2, \\
            & x^1 \geq x^0, \; x^1 \geq 0, \; x^1\left(x^0 - x^1\right) \geq 0.
    \end{array}
\end{equation*}
Since \eqref{eq:general_sdp_primal} is the primal SDP relaxation
and \eqref{eq:general_sdp_dual} is the dual SDP relaxation,
we now call \eqref{eq:primal_deepsdp_oneneuron}  a primal problem
and DeepSDP~\eqref{eq:deepsdp_oneneuron}  a dual problem. 

Let $\e \in \Real^3$ such that $\|\e\| = 1$.
For $\u^1, \v^1 \in \Real^3$, we let $x^0 = \trans{\e}\u^1$ and $x^1 = \trans{\e}\v^1$.
We substitute 
$$
    \begin{bNiceMatrix}[vlines=2,hlines=2,margin]
        1 & x^0 & x^1 \\
        x^0 & \Block{2-2}{X} & \\
        x^1 & & 
    \end{bNiceMatrix}
    = \begin{bNiceMatrix}[vlines=2,hlines=2,margin]
        \trans{\e}\e   & \trans{\e}\u^1 & \trans{\e}\v^1 \\
        \trans{\e}\u^1 & \trans{\left(\u^1\right)}\u^1 & \trans{\left(\u^1\right)}\v^1 \\
        \trans{\e}\v^1 & \trans{\left(\v^1\right)}\u^1 & \trans{\left(\v^1\right)}\v^1
    \end{bNiceMatrix},
$$
 into \eqref{eq:primal_deepsdp_oneneuron}    
as shown in \cref{ssec:tight_rank_constrained_SDP}.
Then, we obtain the following nonlinear optimization problem:
\begin{equation} \label{eq:nc_interpretation_for_oneneuron}
    \begin{array}{rl}
        \min\limits_{\u^1, \v^1} & c\trans{\e}\v^1 \\
        \subto
             & \trans{\e}\v^1 \geq \trans{\e}\left(\u^1 + b^0\e\right), \quad \trans{\e}\v^1 \geq 0, \quad
                \left\|\v^1\right\|_2^2 \leq \trans{\left(\u^1 + b^0\e\right)}\v^1, \\
             & \|\u^1 - \hat{x}\e\|_2^2 \leq \rho,
    \end{array}
\end{equation}
where $\hat{x} = \left(\underline{x} + \bar{x}\right) / 2$,
and $\rho = \hat{x}^2 - 2\underline{x}\bar{x}$.
The problems \eqref{eq:primal_deepsdp_oneneuron} and \eqref{eq:nc_interpretation_for_oneneuron}
    correspond to \eqref{eq:general_rcsdp} and \eqref{eq:general_nc_interpretation}, respectively.
The tightness of \eqref{eq:deepsdp_oneneuron} can be determined
by testing whether
    all the optimal solutions~\eqref{eq:nc_interpretation_for_oneneuron} are collinear with $\e$.
   
Analyzing   directly
the collinearity of all optimal solutions is, however, challenging.
To address this difficulty, we  decompose \eqref{eq:nc_interpretation_for_oneneuron} into two-stage problems,
    and then combine their optimal solutions.
Using the approach in \cite{Zhang2020},
    the first-stage problem of the decomposed problem can be described as
\begin{equation} \label{eq:oneneuron_first_stage}
    \begin{array}{rl}
        \min\limits_{\v^1} & c\trans{\e}\v^1 \\
        \subto
             & \trans{\e}\v^1 \geq 0, \quad \Phi(\v^1) \leq \rho,
    \end{array}
\end{equation}
and the second-stage problem
\begin{equation} \label{eq:oneneuron_second_stage}
    \begin{array}{rl}
        \Phi(\z) \coloneqq \min\limits_{\u^1} & \|\u^1 - \hat{x}\e\|_2^2 \\
        \subto
             & \trans{\e}\z \geq \trans{\e}\left(\u^1 + b^0\e\right), \quad \left\|\z\right\|_2^2 \leq \trans{\left(\u^1 + b^0\e\right)}\z.
    \end{array}
\end{equation}
We note that the variables of the two-stage problem are different from those in \cite{Zhang2020}.
Specifically, the variable of the first-stage problem~\eqref{eq:oneneuron_first_stage} is $x^1$ (the
output of NN), while that in \cite{Zhang2020} is $x^0$ (the input of NN).

\subsection{Analyzing the tightness}

\begin{lemma} \label{lem:value_second_stage_lemma}
    Suppose that $\e \in \Real^p$ satisfies $\|\e\| = 1$,
        $\z \in \Real^p$ is a feasible point of \eqref{eq:oneneuron_first_stage}, $\hat{x} \geq 0$, $b^0 = 0$,
        and the feasible set of \eqref{eq:oneneuron_second_stage} is nonempty.
    Then, $\left(\u^1\right)^* \coloneqq \z$ is a solution of \eqref{eq:oneneuron_second_stage}.
    In addition, $\Phi(\z) = \left\|\z - \hat{x}\e\right\|^2$ follows.
\end{lemma}
\begin{proof}
    Let $\FC$ be the feasible set of \eqref{eq:oneneuron_second_stage}.
    Without loss of generality, we may assume that $\e = \e^1$.
    Using this and $b^0 = 0$, the first constraint in
    \eqref{eq:oneneuron_second_stage} is 
    $z_1 \ge u_1^1$,
    where $z_1$ is the first element of $\z$.
    Thus, \eqref{eq:oneneuron_second_stage} can be equivalently rewritten as
    \begin{equation} \label{eq:oneneuron_second_stage_e1}
        \begin{array}{rl}
            \min\limits_{\u^1} & \|\u^1 - \hat{x}\e^1\|_2^2 \\
            \subto
                 & z_1 \geq u_1^1, \quad \left\|\z\right\|_2^2 \leq \trans{\left(\u^1\right)}\z.
        \end{array}
    \end{equation}
    It  suffices to show that $\z$ is an optimal solution of \eqref{eq:oneneuron_second_stage_e1}.    
    For the boundaries of two constraints in \eqref{eq:oneneuron_second_stage_e1},
        we define $H_1 \coloneqq \{ z_1\} \times \Real^{p-1}$
        and 
        $H_2 \coloneqq \left\{ \u^1 \in \Real^{p} \,\middle|\, \left\|\z\right\|_2^2 = \trans{\left(\u^1\right)}\z \right\}$
        where $z_1$ is the first element of $\z$.
    From the constraint $\left\|\z\right\|_2^2 = \trans{\left(\u^1\right)}\z$ in $H_2$,
        the hyperplane $H_2$ is perpendicular to $\z$ at $\z$ unless $\z = \0$.

    We have $\trans{\e}\z \geq 0$ from the constraint $\trans{\e}\v^1 \geq 0$ in \eqref{eq:oneneuron_first_stage},
    which indicates $z_1 \ge 0$.
    On the other hand, $|z_1| \le \left\|\z\right\|_2^2$ holds generally.
    Therefore, $0 \leq z_1 \leq \left\|\z\right\|_2$ follows.
    The proof is presented for three cases, depending on the value of $z_1$.

    First, suppose that $z_1 = \|\z\|_2$.
    This implies that there exists $\lambda \in \Real_+$ such that $\z = \lambda\e^1$.
    If $\lambda = 0$, then $\z = \0$.
    While the second constraint $\left\|\z\right\|_2^2 \leq \trans{\left(\u^1\right)}\z$ vanishes when $\z = \0$,
        the first constraint $z_1 \ge u_1^1$ requires $u_1^1 \le 0$, thus $\u^1 \in H_1^- := \left(-\Real_+\right) \times \Real^{p-1}$. Here $H_1^-$ is the half space below $H_1$.
    Since the vector $\e^1$ is perpendicular to $H_1^-$ at the origin $O$ and $\hat{x} \ge 0$, the point $\0$ is closest to $\hat{x}\e^1$ in $H_1^-$.
    Thus, $(\u^1)^* = \z = \0$ under $\lambda = 0$.
    For $\lambda > 0$, from the second constraint $\left\|\z\right\|_2^2 \leq \trans{\left(\u^1\right)}\z$ of \eqref{eq:oneneuron_second_stage_e1},
        we obtain $\trans{\z}\left(\lambda\e^1\right) \leq \trans{\left(\u^1\right)}\lambda\e^1$, and this is 
        equivalent to $z_1 \le u_1^1$.
    Combining  with the first constraint, 
        we deduce that the feasible set of \eqref{eq:oneneuron_second_stage} is
        $\left\{ \u^1 \in \Real^p \,\middle|\, u^1_1 = z_1 \right\}$ which is equal to $H_1$.
    Since $\hat{x}\e^1$ is perpendicular to $H_1$ at $\z$,
        the point $\z$ is an optimal solution $(\u^1)^*$ of \eqref{eq:oneneuron_second_stage_e1}.

    Second, suppose $z_1 = 0$ with $\z \neq \0$.
    Then, the first constraint requires that $\u^1 \in H_1^-$, which is the half space below $H_1$. 
    From the second constraint and $z_1 = 0$, $\FC$ is a subset of
    $$
        H_2^+ \coloneqq \left\{ \u^1 \in \Real^{p} \,\middle|\, u^1_1: \text{free variable}, \; \sum_{i=2}^p z_i\left(u^1_i - z_i\right) \geq 0 \right\},
    $$
    which is the half-space above $H_2$.
    Then, for any point $\u^1 \in \FC \subseteq H_1^- \cap H_2^+$, we can compute the objective function of \eqref{eq:oneneuron_second_stage_e1} as
    \begin{equation} \label{eq:oneneuron_obj}
        \left\| \u^1 - \hat{x}\e^1 \right\|_2^2
        = \left| u^1_1 - \hat{x} \right|^2 + \sum_{i = 2}^p \left| u^1_i \right|^2
        = \left| u^1_1 - \hat{x} \right|^2 + \left\| \u^1_{\{2,\ldots,p\}} \right\|_2^2.
    \end{equation}
    Thus, the first element of $\u^1$ and the other elements can be separately determined to minimize \eqref{eq:oneneuron_second_stage_e1}.
    Since $\hat{x} \geq 0$ and $0 = z_1 \geq u^1_1$, the term $\left| u^1_1 - \hat{x} \right|^2$  of \eqref{eq:oneneuron_obj} attains the smallest value when $u^1_1 = 0$.
    The last term $\left\| \u^1_{\{2,\ldots,p\}} \right\|_2^2$ is to find a point $\u^1$ nearest to the origin $O$ on $H_1 \cap H_2^+$.
    Since $H_2$ is perpendicular to $\z$ at $\z \in H_1 \cap H_2^+$, the nearest point to the origin $O$ is also $\z$.
    Hence, $\z$ is an optimal solution (see Figure~\ref{fig:projection_in_proof}(a)).
    
    Lastly, suppose that $0 < z_1 < \|\z\|$.
    Then, $z_2^2 + \cdots + z_n^2 > 0$ holds.
    We also have $\hat{x}\e^1 \not\in \FC$, because 
    if $\u^1 = \hat{x}\e^1$ satisfies the first constraint $z_1 \ge u_1^1$, then $z_1 \geq \hat{x}$ and
    $$
        \left\|\z\right\|_2^2 - \trans{\left(\hat{x}\e^1\right)}\z = z_1\left(z_1 - \hat{x}\right) + z_2^2 + \cdots + z_n^2 > 0,
    $$
    which implies that the second constraint $\left\|\z\right\|_2^2 \leq \trans{\left(\u^1\right)}\z$ does not hold.
    We now consider the projection of $\hat{x}\e^1$ onto $\FC$ in \eqref{eq:oneneuron_second_stage_e1}.
    For any $\u^1 \in \FC$, we have
    \begin{align*}
        \trans{\left(\u^1 - \z\right)}\left(\hat{x}\e^1 - \z\right)
            &= \trans{\left(\u^1 - \z\right)}\hat{x}\e^1 - \trans{\left(\u^1 - \z\right)}\z \\
            &\leq \hat{x}\left(u^1_1 - z_1\right)
            \leq 0. 
    \end{align*}
    The first inequality follows from the second constraint $\left\|\z\right\|_2^2 \leq \trans{\left(\u^1\right)}\z$,
        and the second inequality follows from the first constraint $z_1 \ge u_1^1$.
    Furthermore, $z_1 \ge z_1$ and 
    $\left\|\z\right\|_2^2 \leq \trans{\left(\z\right)}\z$ ensure
    $\z \in \FC$.
    By the second projection theorem~\cite[Theorem~9.8]{Beck2014},
        $\z$ is the projection of $\hat{x}\e$ onto $\FC$, therefore, $\z$ is the optimal solution on $\FC$
        (see Figure~\ref{fig:projection_in_proof}(b)).
\end{proof}

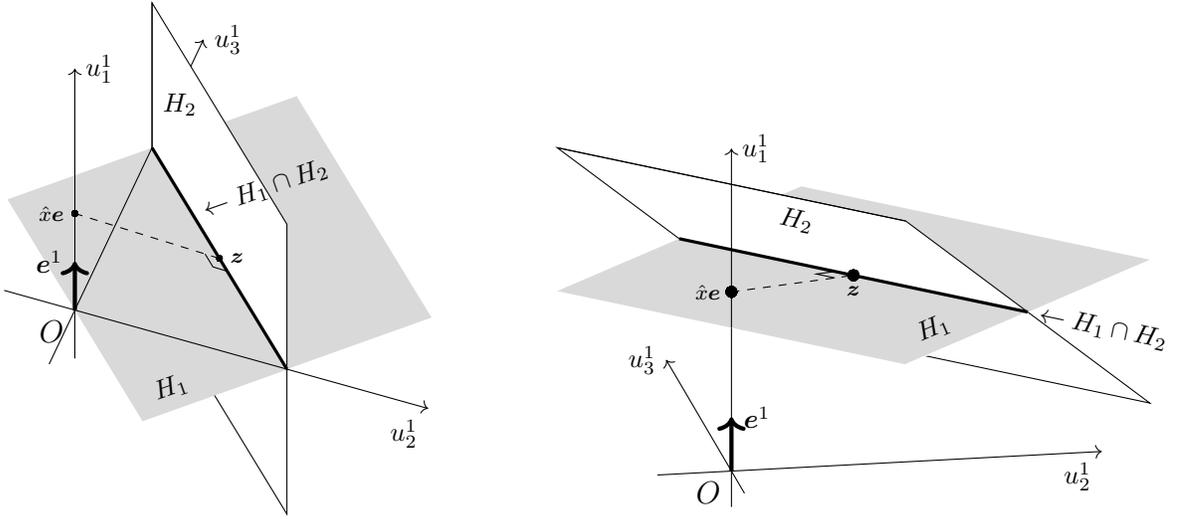
\begin{figure}[!t]
    \centering
    \begin{minipage}{0.47\linewidth}
        \centering
        \tdplotsetmaincoords{40}{20}
        \begin{tikzpicture}[tdplot_main_coords,scale=1.0]
            \coordinate (O) at (0,0,0);
            \coordinate (Z) at (1.5,1.5,0);
            \coordinate (A) at (1.5,4.5,0);
            \coordinate (B) at (4.5,1.5,0);
            \coordinate (C) at (1.5,-1.5,0);
            \coordinate (D) at (-1.5,1.5,0);

            \coordinate (AB) at (1.5-1.5,1.5+1.5,0);
            \coordinate (CD) at (1.5+1.5,1.5-1.5,0);

            \coordinate (E) at ($(AB) + (0, 0, 3)$);
            \coordinate (F) at ($(AB) + (0, 0,-3)$);
            \coordinate (G) at ($(CD) + (0, 0,-3)$);
            \coordinate (H) at ($(CD) + (0, 0, 3)$);

            \fill[white] (AB) -- (F) -- (G) -- (CD) -- (AB);
            \draw[black] (AB) -- (F) -- (G) -- (CD);
            \fill[gray!30!white] (A) -- (B) -- (C) -- (D) -- (A);
            \draw[gray!30!white] (CD) -- (C) node[black,above,sloped,pos=0.75]{\footnotesize$H_1$};
            
            \draw[thin,->] (0,0,0) -- (0,5,0) node [right] {\footnotesize$u^1_3$};      
            
            \fill[white] (AB) -- (E) -- (H) -- (CD) -- (AB);
            \draw[black] (AB) -- (E) -- (H) -- (CD) -- (AB);
            \draw[black] (AB) -- (E) node[black,right,sloped,pos=0.3,rotate=-90]{\footnotesize$H_2$};

            \draw[thin,->] (0,0,0) -- (5,0,0) node [below left] {\footnotesize$u^1_2$}; 
            \draw[thin] (0,0,0) -- (-1,0,0);                                            
            \draw[thin] (0,0,0) -- (0,-1,0);                                            
            \draw[thin,->] (0,0,-1) -- (0,0,5) node [right] {\footnotesize$u^1_1$};   
            \draw[ultra thick,->] (0,0,0) -- (0,0,1) node [left] {\footnotesize$\e^1$};

            \draw[black,very thick] (AB) -- (CD) node[black,right,sloped,pos=0.3,rotate=78]{\footnotesize$\leftarrow H_1 \cap H_2$};
            \draw plot [mark=*, mark size=1] (Z) node [right] {\scriptsize$\z$};
            
            \tdplotsetcoord{xhat}{2}{0}{90}
            \draw plot [mark=*, mark size=1] (xhat) node [left] {\scriptsize$\hat{x}\e$};
            \draw[dashed] (xhat) -- (Z);
            \pic[draw,angle radius=2mm]{right angle=xhat--Z--CD};

            \draw (O) node[below left] {$O$};
        \end{tikzpicture}
        \subcaption{
            $\trans{\left(\e^1\right)}\z = 0$ case.
            The feasible set $\FC$ is the set of points on the gray plane and to the right of the white plane.
        }
    \end{minipage}
    \hfill
    \begin{minipage}{0.49\linewidth}
        \centering
        \tdplotsetmaincoords{72.5}{-10}
        \vspace{8.9ex}
        \begin{tikzpicture}[tdplot_main_coords,scale=1.0]
            \coordinate (Z) at (2,2,2);
            \coordinate (A) at ( 6, 2,2);
            \coordinate (B) at ( 2,-2,2);
            \coordinate (C) at (-2, 2,2);
            \coordinate (D) at ( 2,6,2);
            \coordinate (E) at (2,-2,4);
            \coordinate (F) at (-2,2,4);
            \coordinate (G) at (2,6,0);
            \coordinate (H) at (6,2,0);
            \coordinate (AB) at ($(A)!0.5!(B)$);
            \coordinate (CD) at ($(C)!0.5!(D)$);
            \coordinate (EF) at ($(E)!0.5!(F)$);
            \coordinate (GH) at ($(G)!0.5!(H)$);
            
            \draw[black] (CD) -- (G) -- (H) -- (AB);
            
            \draw[thin,->] (0,0,0) -- (5,0,0) node [below left] {\footnotesize$u^1_2$}; 
            \draw[thin] (0,0,0) -- (-1,0,0);                                            
            \draw[thin,->] (0,0,0) -- (0,5,0) node [left] {\footnotesize$u^1_3$};      
            \draw[thin] (0,0,0) -- (0,-1,0);                                            
            \draw[thin] (0,0,-0.5) -- (0,0,2);   
            \draw[ultra thick,->] (0,0,0) -- (0,0,0.75) node [right] {\footnotesize$\e^1$};
            
            \fill[gray!30!white] (A) -- (B) -- (C) -- (D) -- (A);
            
            \fill[white] (CD) -- (F) -- (E) -- (AB) -- (CD);
            \draw[black] (CD) -- (F) -- (E) -- (AB);
            
            
            \draw[thin,->] (0,0,2) -- (0,0,4.5) node [right] {\footnotesize$u^1_1$};   
            
            \draw[white] (AB) -- (B) node[black,above,sloped,pos=0.7]{\footnotesize$H_1$};
            \draw[black] (E) -- (F) node[black,below,sloped,pos=0.3]{\footnotesize$H_2$};
            \draw[black,very thick] (AB) -- (CD) node[black,right,sloped,pos=0.0,rotate=0]{\footnotesize$\leftarrow H_1 \cap H_2$};
            
            \draw plot [mark=*, mark size=2] (Z) node [below] {\scriptsize$\z$}; 

            \tdplotsetcoord{xhat}{2.5}{0}{90}
            \draw plot [mark=*, mark size=2] (xhat) node [left] {\scriptsize$\hat{x}\e$};
            \draw[dashed] (xhat) -- (Z);
            \pic[draw,angle radius=2.5mm]{right angle=xhat--Z--CD};
            
            \draw (O) node[below left] {$O$};
        \end{tikzpicture}
        \subcaption{
        $0 < \trans{\left(\e^1\right)}\z < \|\z\|$ case.
        The feasible set $\FC$ is the set of points below the gray plane and above the white plane.
        }
    \end{minipage}
    \caption{%
        Projection of $\hat{x}\e$ to $\z$ in the case $p = 3$.
        The gray and white planes denote $H_1$ and $H_2$, respectively. 
        The bold line represents the intersection of them.
        }
    \label{fig:projection_in_proof}    
\end{figure}

We note that the result above can be extended to a more general case 
$b^0 \neq 0$ by shifting the variables.
\begin{prop} \label{prop:value_second_stage}
    Suppose that $\e \in \Real^p$ satisfies $\|\e\| = 1$, $\z \in \Real^p$ is a feasible point of \eqref{eq:oneneuron_first_stage},
        $\hat{x} \geq -b^0$,
        and the feasible set of \eqref{eq:oneneuron_second_stage} is nonempty.
    Then, $\left(\u^1\right)^* \coloneqq \z - b^0\e$ is a solution of \eqref{eq:oneneuron_second_stage}.
    In addition, $\Phi(\z) = \left\|\z - \left(b^0 + \hat{x}\right)\e\right\|^2$ follows.
\end{prop}
\begin{proof}
    Let $\tilde{\u}^1 \coloneqq \u^1 + b^0\e$ and $\tilde{x} \coloneqq \hat{x} + b^0 \geq 0$.
    Then, \eqref{eq:oneneuron_second_stage} is rewritten as
    $$
        \begin{array}{rl}
            \min\limits_{\tilde{\u}^1} & \|\tilde{\u}^1 - \tilde{x}\e\|_2^2 \\
            \subto
                 & \trans{\e}\z \geq \trans{\e}\tilde{\u}^1, \quad \left\|\z\right\|_2^2 \leq \trans{\left(\tilde{\u}^1\right)}\z.
        \end{array}
    $$
    Since $\tilde{x} \geq 0$,
        the above problem has an optimal solution $\left(\tilde{\u}^1\right)^* \coloneqq \z$ by Lemma~\ref{lem:value_second_stage_lemma}.
    By definition of $\tilde{\u}^1$,
        the point $\z - b^0\e$ is a solution $(\u^1)^*$ of \eqref{eq:oneneuron_second_stage}.
\end{proof}
\noindent

Using Proposition~\ref{prop:value_second_stage},
    the first-stage problem~\eqref{eq:oneneuron_first_stage} can be written as
\begin{equation} \label{eq:oneneuron_first_stage_revised}
    \begin{array}{rl}
        \min\limits_{\v^1} & c\trans{\e}\v^1 \\
        \subto
             & \trans{\e}\v^1 \geq 0, \quad \|\v^1 - \hat{x}\e\|_q \leq \rho.
    \end{array}
\end{equation}
The objective function and the left-hand side of the first constraint are parallel in the direction 
determined with $\trans{\e}\v^1$.
Thus, it is easy to solve \eqref{eq:oneneuron_first_stage_revised},
    and there exists a solution $\v^1$ which is collinear with $\e$.
\begin{theorem} \label{thm:tightness_oneneuron}
    Let $\e$ be an arbitrary unit vector with $\|\e\| = 1$.
    Suppose $\hat{x} \geq -b^0$.
    There exists an optimal solution $((\u^1)^*, (\v^1)^*)$ of \eqref{eq:nc_interpretation_for_oneneuron} such that
        the vectors $(\u^1)^*$ and $(\v^1)^*$ are collinear with $\e$.
    Thus, DeepSDP~\eqref{eq:primal_deepsdp_oneneuron} has a rank-1 solution,
        and it is a tight relaxation of the corresponding QCQP
        \eqref{eq:deepsdp_oneneuron}.
\end{theorem}
\begin{proof}
    Without loss of generality, we may assume that $\e = \e^1$.
    Then, the objective function and the first constraint depend only on the value of $v^1_1$.
    The point $\trans{[\hat{x} + \rho, \trans{\0}]}$ attains the maximum value of $v^1_1$ on the feasible set when $c>0$,
    while $\trans{[\hat{x} - \rho, \trans{\0}]}$ attains the minimum value.
    It suffices to consider the points $\v^1$ on
        the line segment between $\trans{[\hat{x} + \rho, \trans{\0}]}$ and $\trans{[\hat{x} - \rho, \trans{\0}]}$.
    Hence, $\v^1$ and $\e^1$ are collinear, {\it i.e.,} $\v^1 = v^1_1 \e^1$.
    By Proposition~\ref{prop:value_second_stage}, $\u^1$ is also collinear with $\e^1$.
    Applying Lemma~\ref{lem:rank_one_solution_collinear}, we conclude that \eqref{eq:primal_deepsdp_oneneuron} has a rank-1 solution.
\end{proof}

\section{Tight DeepSDP for a single-layer neural network} \label{sec:deepsdp_onelayer}

We discuss the tightness of DeepSDP~\eqref{eq:introduction_deepsdp}  for a single-layer NN ($L = 1$). 
More precisely,
$$
    f(\x^0) = W^1\phi\left(W^0 \x^0 + \b^0\right) + \b^1.
$$
In this case, the matrix functions in DeepSDP~\eqref{eq:introduction_deepsdp} can be rewritten as
\begin{align*}
    M_\mathrm{in}(P) &= \trans{\begin{bmatrix} 1 & \trans{\0} & \trans{\0} \\ \0 & I_{n_0} & O \end{bmatrix}} P \begin{bmatrix} 1 & \trans{\0} & \trans{\0} \\ \0 & I_{n_0} & O \end{bmatrix} \in \SymMat^{1 + n_0 + n_1}, \\
    M_\mathrm{mid}(Q) &= \trans{\begin{bmatrix} 1 & \trans{\0} & \trans{\0} \\ \b^0 & W^0 & O \\ \0 & O & I_{n_1} \end{bmatrix}} Q \begin{bmatrix} 1 & \trans{\0} & \trans{\0} \\ \b^0 & W^0 & O \\ \0 & O & I_{n_1} \end{bmatrix} \in \SymMat^{1 + n_0 + n_1}, \\
    M_\mathrm{out}(S) &= \trans{\begin{bmatrix} 1 & \trans{\0} & \trans{\0} \\ \0 & I_{n_0} & O \\ \b^1 & O & W^1 \end{bmatrix}} S \begin{bmatrix} 1 & \trans{\0} & \trans{\0} \\ \0 & I_{n_0} & O \\ \b^1 & O & W^1 \end{bmatrix} \in \SymMat^{1 + n_0 + n_1}.
\end{align*}
We introduce the following assumptions.
\begin{assum} \label{asm:polytope_safety_set}
    The safety specification set $S_y$ is a polytope.
\end{assum}
\begin{assum} \label{asm:no_effect_lastlayer}
    The last layer is the identity layer, i.e., $W^1=I$ and $\b^1 = \0$.
\end{assum}

\noindent
Assumption~\ref{asm:no_effect_lastlayer} may appear to impose a strong restriction on the DeepSDPs considered,
    and potentially limit  the generalization of the analysis.
However, it does not  change the class of the DeepSDPs defined under Assumption~\ref{asm:polytope_safety_set}.
When $S_y$ is obtained from the DeepSDP with $W_1 = I$,
    the safety specification set of the original DeepSDP is the projection of $S_y$ by the original $W^1$, {\it i.e.,}
$$
    \left\{ W^1 y + \b_1 \,\middle|\, y \in S_y \right\},
$$
which is also a polytope.
In addition, Assumption~\ref{asm:no_effect_lastlayer} induces $M_\mathrm{out}(S) = S$.

The repeated nonlinearities~\eqref{eq:sloperestriction_relu_yy}, described in \cref{sssec:repeated_nonlinearities},
    are redundant constraints in the corresponding QCQP~\eqref{eq:qcqp_for_deepsdp}.
Thus, if the SDP relaxation of \eqref{eq:qcqp_for_deepsdp} without \eqref{eq:sloperestriction_relu_yy} is tight, 
    then that of \eqref{eq:qcqp_for_deepsdp} with \eqref{eq:sloperestriction_relu_yy} is also tight.
We impose the following assumption in the proofs to show the tightness of DeepSDP~\eqref{eq:introduction_deepsdp}.
\begin{assum} \label{asm:no_repeated_inequalities}
DeepSDP~\eqref{eq:introduction_deepsdp} has no repeated nonlinearities.
\end{assum}

In \cref{ssec:polytope_safety_set}, we describe the matrix set $\SC$ according to \cref{asm:Sy_poly}.
Subsequently, we discuss the tightness of DeepSDP~\eqref{eq:introduction_deepsdp}
    if $\XC$ is an ellipsoid in \cref{ssec:ellipsoid_input_onelayer_deepsdp} and a rectangle in \cref{ssec:box_input_onelayer_deepsdp}.

\subsection{Safety specification sets for polytopes} \label{ssec:polytope_safety_set}

We discuss the safety specification set $S_y$ for estimating a polytope.
Since any polytope can be represented by the intersection of a finite number of half spaces,
we express $S_y$ with
\begin{equation} \label{eq:polytope_inequality_safeset}
    S_y = 
    \bigcap\limits_{\ell=1}^M \left\{ \y \in \Real^{n_L+1} \,\middle|\,
        \trans{(\c^{\ell})}\y - d_{\ell} \geq 0
        \right\},
\end{equation}
where $\c^1, \ldots, \c^M \in \Real^{n_{L+1}}$ and $d_1, \dots, d_M \in \Real$.
The objective of DeepSDP~\eqref{eq:introduction_deepsdp} in this section is to optimize $d_1, \dots, d_M$ 
 by fixing the half-space directions $\c^1,\dots, \c^M$, 
 in order to sufficiently minimize $S_y$. 
Since the $\ell$th half-space depends only on $\c^\ell$ and $d_\ell$,
    each half-space consisting of \eqref{eq:polytope_inequality_safeset} can be separately considered. 
An inequality of the form $\trans{\c}\y -d \geq 0$ associated with the half space can be rewritten as
\begin{equation} \label{eq:inequality_halfspace_coefficient}
    \trans{\begin{bmatrix} 1 \\ \x^0 \\ \y \end{bmatrix}}
        \begin{bmatrix} -2d & \trans{\0} & \trans{\c} \\  \0 & O & O \\ \c & O & O \end{bmatrix}
        \begin{bmatrix} 1 \\ \x^0 \\ \y \end{bmatrix} \geq 0.
\end{equation}
The polytope-shaped safety specification set can be obtained from all the coefficient matrices of the above inequality:
\begin{equation}
    \SC = \left\{ \begin{bmatrix} -2d & \trans{\0} & \trans{\c} \\  \0 & O & O \\ \c & O & O \end{bmatrix}
        \,\middle|\, d \in \Real
    \right\}.
\end{equation}
Using this $\SC$, the resulting DeepSDP is 
\begin{equation} \label{eq:deepsdp_onelayer}
    \begin{array}{rl}
        \max\limits_{P, \boldsymbol{\lambda}, \boldsymbol{\nu}, \boldsymbol{\eta}, \boldsymbol{\mu}, d} & 2d \\
        \subto
        & R \coloneqq
            M_\mathrm{in}(P) +
        \begin{bmatrix}
            0 & \trans{\boldsymbol{\nu}}W^0 & - \trans{\boldsymbol{\nu}} - \trans{\boldsymbol{\eta}} \\
            \trans{(W^0)}\boldsymbol{\nu} & O & -\trans{(W^0)}\diag(\boldsymbol{\lambda}) \\
            - \boldsymbol{\nu} - \boldsymbol{\eta} & -\diag(\boldsymbol{\lambda})W^0 & 2\diag(\boldsymbol{\lambda})
            \end{bmatrix} \\
        & \qquad +
          \begin{bmatrix} 0 & \trans{\0} & \trans{\0} \\  \0 & O & - \trans{(W^0)}T \\ \0 & - T W^0 & 2T \end{bmatrix} +
          \begin{bmatrix} -2d & \trans{\0} & \trans{\c} \\  \0 & O & O \\ \c & O & O \end{bmatrix} \succeq O, \\
        & T = \sum\limits_{i = 1}^n \sum\limits_{j = i + 1}^n \mu_{ij}(\e_i - \e_j)\trans{(\e_i - \e_j)}, \quad \mu_{ij} \geq 0, \\
        & P \in \PC_\XC, \quad \boldsymbol{\lambda}, \boldsymbol{\nu}, \boldsymbol{\eta} \in \Real_+^n, \quad d \in \Real,
    \end{array}
\end{equation}
where the objective is to maximize $2d$ as we want to find the largest $d$  
    such that $\trans{\c}\x^1 \geq d$ holds for all $\x^1 \in \XC_1$.
Notice that  \eqref{eq:deepsdp_onelayer} is not an standard SDP due to  the matrix variable $P$.  
In \cref{ssec:ellipsoid_input_onelayer_deepsdp,ssec:box_input_onelayer_deepsdp},  we show that
the precise formulation for $\PC_\XC$ provides an SDP from
\eqref{eq:deepsdp_onelayer}.

The dual problem of \eqref{eq:deepsdp_onelayer} is
\begin{equation} \label{eq:primal_for_deepsdp_onelayer}
    \begin{array}{rll}
        \min\limits_{\substack{\x^0,\x^1, \\ X^{00},X^{10},X^{11}}} & 2\trans{\c}\x^1 \\
        \subto
             & \ip{M_\mathrm{in}(P)}{G} \leq 0, \; \forall P \in \PC_\XC, \\
             & \ip{\begin{bmatrix}
                    0 & \trans{\0} & -b^0_i \trans{\left(\e^i\right)} \\
                    \0 & O & -\trans{\left(W^0\right)}\e_i\trans{\e_i} \\
                    -b^0_i \e^i & -\e_i\trans{\e_i}W^0 & 2\e_i\trans{\e_i}
                \end{bmatrix}}{G} \leq 0, & i = 1,\ldots,n_1, \\
             & \ip{\begin{bmatrix}
                    2b^0_i & \trans{\e_i}W^0 & -\trans{\e_i} \\
                    \trans{\left(W^0\right)}\e_i & O & O \\
                    -\e_i & O & O
                \end{bmatrix}}{G} \leq 0, & i = 1,\ldots,n_1, \\
             & \ip{\begin{bmatrix}
                    0 & \trans{\0} & -\trans{\e_i} \\
                    \0 & O & O \\
                    -\e_i & O & O
                \end{bmatrix}}{G} \leq 0, & i = 1,\ldots,n_1, \\
             & G \coloneqq \begin{bmatrix}
                1 & \trans{\left(\x^0\right)} & \trans{\left(\x^1\right)} \\
                \x^0 & X^{00} & \trans{\left(X^{10}\right)} \\
                \x^1 & X^{10} & X^{11}
                \end{bmatrix} \in \SymMat_+^{(1+n_0+n_1)},
    \end{array}
\end{equation}
which corresponds to the primal SDP relaxation \eqref{eq:sdp_for_deepsdp} in \cref{ssec:between_tightness_and_deepsdp}.
The repeated nonlinearities~\eqref{eq:sloperestriction_relu_yy} 
are removed here
by Assumption~\ref{asm:no_repeated_inequalities},
thus the constraints corresponding to the dual variables $\mu_{ij}$ and $T$ in \eqref{eq:deepsdp_onelayer}
    do not appear in \eqref{eq:primal_for_deepsdp_onelayer}.
For the fixed vector $\e \in \Real^p$ such that $\|\e\| = 1$,
    define $\u^1,\ldots,\u^{n_0} \in \Real^p$ and $\v^1,\ldots,\v^{n_1} \in \Real^p$
    to substitute $\x^0$ and $\x^1$ with
\begin{gather*}
    \x^0 = \begin{bmatrix} \trans{\e}\u^1 \\ \vdots \\ \trans{\e}\u^{n_0} \end{bmatrix} \in \Real^{n_0}, \quad
    \x^1 = \begin{bmatrix} \trans{\e}\v^1 \\ \vdots \\ \trans{\e}\v^{n_1} \end{bmatrix} \in \Real^{n_1}, \quad
    X^{00} = \begin{bNiceMatrix}[margin]
        \trans{\left(\u^1\right)}\u^1 & \Cdots & \trans{\left(\u^1\right)}\u^{n_0} \\
        \Vdots & \Ddots & \Vdots \\
        \trans{\left(\u^{n_0}\right)}\u^1 & \cdots & \trans{\left(\u^{n_0}\right)}\u^{n_0}
    \end{bNiceMatrix} \in \SymMat^{n_0}, \\
    X^{10} = \begin{bNiceMatrix}[margin]
        \trans{\left(\v^1\right)}\u^1 & \Cdots & \trans{\left(\v^1\right)}\u^{n_0} \\
        \Vdots & \Ddots & \Vdots \\
        \trans{\left(\v^{n_1}\right)}\u^1 & \cdots & \trans{\left(\v^{n_1}\right)}\u^{n_0}
    \end{bNiceMatrix} \in \Real^{n_1 \times n_0}, \quad
    X^{11} = \begin{bNiceMatrix}[margin]
        \trans{\left(\v^1\right)}\v^1 & \Cdots & \trans{\left(\v^1\right)}\v^{n_1} \\
        \Vdots & \Ddots & \Vdots \\
        \trans{\left(\v^{n_1}\right)}\v^1 & \cdots & \trans{\left(\v^{n_1}\right)}\v^{n_1}
    \end{bNiceMatrix} \in \SymMat^{n_1}.
\end{gather*}
Then, from \eqref{eq:primal_for_deepsdp_onelayer}, we have the following nonlinear formulation:
\begin{subequations} \label{eq:nc_interpretation_for_onelayer}
    \begin{alignat}{2}
        \min\limits_{\u^j,\v^i} \ \ & {\textstyle 2\sum_{i=1}^{n_1} c_i \, \trans{\e}\v^i} \label{eq:nc_interpretation_for_onelayer_objective} \\
        \subto \ \
             & \trans{\e}\v^i \geq 0, \label{eq:nc_interpretation_for_onelayer_nonnegative} \\
             & \trans{\e}\v^i \geq {\textstyle \trans{\e}\left(\sum_{j=1}^{n_0} W_{ij}\u^j + b^0_i\e \right)}, \quad & & i = 1,\ldots,n_1, \label{eq:nc_interpretation_for_onelayer_value} \\
             & \left\|\v^i\right\|_2^2 \leq {\textstyle \trans{\left(\sum_{j=1}^{n_0} W_{ij}\u^j + b^0_i\e \right)}\v^i}, \quad & & i = 1,\ldots,n_1, \label{eq:nc_interpretation_for_onelayer_eithercondition} \\
             & \ip{M_\mathrm{in}(P)}{G} \leq 0, \; \forall P \in \PC_\XC. \label{eq:nc_interpretation_for_onelayer_inputset}
    \end{alignat}
\end{subequations}
The matrix $G$ still exists in \eqref{eq:nc_interpretation_for_onelayer_inputset}.
The transformation of this constraint depends on the definition of $\XC$,
    and thus it will be discussed in the subsequent subsections.


\subsection{Ellipsoidal input} \label{ssec:ellipsoid_input_onelayer_deepsdp}

Consider the case that the input set $\XC$ is an ellipsoid with the center $\hat{\x} \in \Real^{n_0}$ and radius $\rho$:
\begin{equation} \label{eq:input_region_ellipsoid}
    \XC \coloneqq \left\{ \x \,\middle|\, \left\| \x - \hat{\x} \right\|_2 \leq \rho \right\}.
\end{equation}
For \eqref{eq:input_region_ellipsoid}, the matrix set $\PC_\XC$ becomes
$$
    \PC_\XC = \left\{
        \gamma \begin{bmatrix}
            \trans{\hat{\x}}\hat{\x} - \rho^2 & -\trans{\hat{\x}} \\
            -\hat{\x} & I
        \end{bmatrix}
        \,\middle|\, \gamma \geq 0,\; \gamma \in \Real \right\},
$$
where $\gamma$ is the dual variable of DeepSDP~\eqref{eq:introduction_deepsdp}.
By  normalizing $ \PC_\XC$ with  $\gamma = 1$, we have
\begin{align*}
    \ip{M_\mathrm{in}\left(\begin{bmatrix}
        \trans{\hat{\x}}\hat{\x} - \rho^2 & -\trans{\hat{\x}} \\
        -\hat{\x} & I
    \end{bmatrix}\right)}{G}
    &= \sum_{j=1}^{n_0} \left\{ \trans{\left(\u^j\right)}\u^j - 2\hat{x}_j\trans{\e}\u^j + \hat{x}_j^2 \trans{\e}\e \right\} - \rho^2 \\
    &= \sum_{j=1}^{n_0} \left\| \u^j - \hat{x}_j \e \right\|_2^2 - \rho^2 \; \leq 0.
\end{align*}
Thus, the problem~\eqref{eq:nc_interpretation_for_onelayer} can be described precisely as 
\begin{equation} \label{eq:nc_interpretation_for_onelayer_ellipsoid}
        \min\limits_{\u^j,\v^i} \ \ \eqref{eq:nc_interpretation_for_onelayer_objective} \quad
        \subto \ \
            \eqref{eq:nc_interpretation_for_onelayer_nonnegative}, \,
            \eqref{eq:nc_interpretation_for_onelayer_value}, \,
            \eqref{eq:nc_interpretation_for_onelayer_eithercondition}, \,
            \sum_{j=1}^{n_0} \|\u^j - \hat{x}_j\e\|_2^2 \leq \rho^2.
\end{equation}
Applying the same decomposition  method  in \cref{sec:oneneuron} results in the following two-stage problem:
\begin{equation} \label{eq:first_stage_for_onelayer_ellipsoid}
    \begin{array}{rl}
        \min\limits_{\v^1,\ldots,\v^{n_1}} & \eqref{eq:nc_interpretation_for_onelayer_objective} \\
        \subto & \eqref{eq:nc_interpretation_for_onelayer_nonnegative}, \quad \Psi_\mathrm{ellipsoid}(\v^1,\ldots,\v^{n_1}) \leq \rho^2,
    \end{array}
\end{equation}
\begin{equation} \label{eq:second_stage_for_onelayer_ellipsoid}
    \begin{array}{rl}
        \Psi_\mathrm{ellipsoid}(\v^1,\ldots,\v^{n_1}) \coloneqq
        \min\limits_{\u^1,\ldots,\u^{n_0}} & \sum\limits_{j=1}^{n_0} \left\|\u^j - \hat{x}_j\e\right\|_2^2 \\
        \subto & \eqref{eq:nc_interpretation_for_onelayer_value},\, \eqref{eq:nc_interpretation_for_onelayer_eithercondition}.
    \end{array}
\end{equation}
The second-stage problem~\eqref{eq:second_stage_for_onelayer_ellipsoid} is a convex optimization problem with
    a quadratic objective function and linear inequality constraints.

The following lemma characterizes solutions to the second-stage problem,
    and shows that all variables are a linear combination of $\e^1$, $\v^1, \ldots, \v^{n_1}$.
\begin{lemma} \label{lem:solution_form_of_second_stage_ellipsoid}
    Suppose that $\e = \e^1$.
    For any optimal solution $\left((\u^1)^*,\ldots,(\u^{n_0})^*\right)$ of \eqref{eq:second_stage_for_onelayer_ellipsoid},
     there exist $\m \in \Real^{n_0}$ and $M \in \Real^{n_1 \times n_0}$ such that 
    $$
        (\u^j)^* = m_j \e^1 + \sum_{i=1}^{n_1} M_{ij} \v^i \quad \text{for each $j \in \left\{ 1,\ldots,n_0 \right\}$}.
    $$
\end{lemma}
\begin{proof}
    Since \eqref{eq:second_stage_for_onelayer_ellipsoid} is a convex optimization,
        any pair $((\tilde{\u}^1, \ldots, \tilde{\u}^{n_0}), (\tilde{\boldsymbol{\nu}}, \tilde{\boldsymbol{\lambda}}))$
        of the primal and dual solutions satisfies the KKT conditions:
    \begin{subequations} \label{eq:KKT_of_second_stage_for_onelayer_ellipsoid}
        \begin{empheq}[left = {\empheqlbrace \quad}, right = {}]{gather}
            \eqref{eq:nc_interpretation_for_onelayer_value},\, \eqref{eq:nc_interpretation_for_onelayer_eithercondition}, \,
                \boldsymbol{\nu} \in \Real_+^{n_1}, \, \boldsymbol{\lambda} \in \Real_+^{n_1}, \\
            \begin{bmatrix} \u^1 \\ \vdots \\ \u^{n_0} \end{bmatrix}
            = \begin{bmatrix} \hat{x}_1\e^1 \\ \vdots \\ \hat{x}_{n_0}\e^1 \end{bmatrix}
            - \sum_{i=1}^{n_1} \frac{\nu_i}{2} \begin{bmatrix} W_{i1}\e^1 \\ \vdots \\ W_{in}\e^1 \end{bmatrix}
            + \sum_{i=1}^{n_1} \frac{\lambda_i}{2} \begin{bmatrix} W_{i1}\v^i \\ \vdots \\ W_{in}\v^i \end{bmatrix} \label{eq:KKT_of_second_stage_for_onelayer_ellipsoid:stationarity} \\
            \nu_i \left[{\textstyle \trans{\left(\e^1\right)}\left(\sum_{j=1}^{n_0} W_{ij}\u^j + b^0_i\e^1 \right)} - \trans{\left(\e^1\right)}\v^i\right] = 0, \quad i = 1,\ldots,n_1, \\
            \lambda_i \left[\left\|\v^i\right\|^2 - {\textstyle \trans{\left(\sum_{j=1}^{n_0} W_{ij}\u^j + b^0_i\e^1 \right)}\v^i}\right] = 0, \quad i = 1,\ldots,n_1.
        \end{empheq}
    \end{subequations}
    By \eqref{eq:KKT_of_second_stage_for_onelayer_ellipsoid:stationarity},
        the desired result follows by taking
        \begin{alignat*}{2}
            m_j    &= \hat{x}_j - \sum_i \frac{\nu_i W_{ij}}{2} \quad & & \text{for $j = 1,\ldots,n_0$}, \\
            M_{ij} &= \frac{\lambda_i W_{ij}}{2} \quad & & \text{for $j = 1,\ldots,n_0, \; i = 1,\ldots,n_1$}.
        \end{alignat*}
\end{proof}

Using Lemma~\ref{lem:solution_form_of_second_stage_ellipsoid},
    the first-stage problem~\eqref{eq:first_stage_for_onelayer_ellipsoid} can be rewritten as
\begin{equation} \label{eq:formulation_with_z_e1}
        \min\limits_{\v^i} \ \ \eqref{eq:nc_interpretation_for_onelayer_objective} \quad
        \subto \ \
            \eqref{eq:nc_interpretation_for_onelayer_nonnegative}, \;
            \sum\limits_{j=1}^{n_1} \left\| m_j \e^1 + \sum\limits_{i=1}^{n_1} M_{ij} \v^i \right\|_2^2 \leq \rho^2.
\end{equation}

\begin{lemma} \label{lem:collinearity_of_z_on_formulation_with_e1}
    Suppose that $\e = \e^1$.
    The problem~\eqref{eq:formulation_with_z_e1} has an optimal solution $\left((\v^1)^*, \ldots, (\v^{n_1})^*\right)$
    such that $(\v^1)^*,\ldots,(\v^{n_1})^*$ are collinear with $\e^1$.
\end{lemma}
\begin{proof}
    Let $(\bar{\v}^1, \ldots, \bar{\v}^{n_1})$ be an optimal solution of \eqref{eq:formulation_with_z_e1}.
    Assume on the contrary  that at least one optimal solution is not collinear with $\e^1$.
    For all $i \in \{1,\ldots,n_1\}$, we define $\hat{\v}^i \coloneqq \trans{\begin{bmatrix} \bar{v}^i_1,\, 0,\, \cdots,\, 0 \end{bmatrix}} \in \Real^n$
        as the projection of $\bar{\v}^i$ onto the set $\left\{ k\e^1 \,\middle|\, k \in \Real\right\}$.
    Then, as $\trans{\left(\e^1\right)}\hat{\v}^i = \trans{\left(\e^1\right)}\bar{\v}^i$ for all $i$,
        $\left(\hat{\v}^1,\ldots,\hat{\v}^{n_1}\right)$ satisfies the constraints~\eqref{eq:nc_interpretation_for_onelayer_nonnegative}
        with the same objective value, {\it i.e.},
    $$
        \sum_{i=1}^{n_1} c_i\trans{\left(\e^1\right)}\hat{\v}^i = \sum_{i=1}^{n_1} c_i\trans{\left(\e^1\right)}\bar{\v}^i.
    $$
    For the last constraint of~\eqref{eq:formulation_with_z_e1},
        from the equivalence of  the first elements of vectors $M_{ij} \hat{\v}^i$ and $M_{ij} \bar{\v}^i$ for any pair $(i, j)$,
        we have
    \begin{align}
        \sum\limits_{j=1}^{n_1} \left\| m_j \e^1 + \sum\limits_{i=1}^{n_1} M_{ij} \hat{\v}^i \right\|_2^2
        & = \sum\limits_{j=1}^{n_1} \left(m_j  + \sum\limits_{i=1}^{n_1} M_{ij} \bar{v}_1^i\right)^2 \nonumber \\ 
        &\leq \sum\limits_{j=1}^{n_1} \left[
            \left(m_j  + \sum\limits_{i=1}^{n_1} M_{ij} \bar{v}_1^i\right)^2 + 
            \sum_{k=2}^{n_1} \left( \sum\limits_{i=1}^{n_1} M_{ij} \bar{v}^i_{k} \right)^2 \right] \nonumber \\
        &= \sum\limits_{j=1}^{n_1}
            \left\| m_j \e^1 + \sum\limits_{i=1}^{n_1} M_{ij} \bar{\v}^i \right\|_2^2 \leq \rho^2. \label{eq:formulation_with_z_e1_inequality_rho}
    \end{align}
    Hence, $\left(\hat{\v}^1,\ldots,\hat{\v}^{n_1}\right)$ is also an optimal solution of \eqref{eq:formulation_with_z_e1}.
    Thus, 
    $\left((\hat{\v}^1)^*, \ldots, (\hat{\v}^{n_1})^*\right)$
        is an optimal solution with collinearity.
\end{proof}

From  Lemmas~\ref{lem:solution_form_of_second_stage_ellipsoid} and 
\ref{lem:collinearity_of_z_on_formulation_with_e1}, associated with each of the two stages,
    we obtain the tightness condition for the single-layer DeepSDP~\eqref{eq:introduction_deepsdp} with the ellipsoidal input set.

\begin{theorem} \label{thm:tightness_onelayer_ellipsoid}
    Suppose that Assumptions~\ref{asm:polytope_safety_set} and \ref{asm:no_effect_lastlayer} hold.
    The problem~\eqref{eq:primal_for_deepsdp_onelayer} has a rank-1 matrix solution.
    In addition, if both \eqref{eq:primal_for_deepsdp_onelayer} and \eqref{eq:introduction_deepsdp} 
    have optimal solutions,
    and the feasible set of \eqref{eq:introduction_deepsdp} is bounded,
    then DeepSDP \eqref{eq:introduction_deepsdp} is a tight relaxation for the original QCQP~\eqref{eq:qcqp_for_deepsdp}.

\end{theorem}
\begin{proof}
    Without loss of generality, we may assume that $\e = \e^1$.
    By Lemma~\ref{lem:solution_form_of_second_stage_ellipsoid},
        the second-stage problem~\eqref{eq:second_stage_for_onelayer_ellipsoid} has an optimal solution:
    $$
        \begin{bmatrix} \left(\u^1\right)^* \\  \vdots \\  \left(\u^{n_0}\right)^* \end{bmatrix} =
        \begin{bmatrix}
            m_1\left(\v^1,\ldots,\v^{n_1}\right) \e^1 + \sum_{i=1}^{n_1} M_{i1}\left(\v^1,\ldots,\v^{n_1}\right) \v^i \\
            \vdots \\
            m_{n_0}\left(\v^1,\ldots,\v^{n_1}\right) \e^1 + \sum_{i=1}^{n_1} M_{in_0}\left(\v^1,\ldots,\v^{n_1}\right) \v^i
        \end{bmatrix}
    $$
    with some $\v^1, \ldots, \v^{n_1}$,
        where $m_j\left(\v^1,\ldots,\v^{n_1}\right)$ and $M_{ij}\left(\v^1,\ldots,\v^{n_1}\right)$ 
        depend on the variable $\v^1, \ldots, \v^{n_1}$. 
    Then, \eqref{eq:first_stage_for_onelayer_ellipsoid} can be rewritten as the form
    $$
        \begin{array}{rl}
            \min\limits_{\v^1,\ldots,\v^{n_1}} & \eqref{eq:nc_interpretation_for_onelayer_objective} \\
            \subto & \eqref{eq:nc_interpretation_for_onelayer_nonnegative},\quad
                     \sum\limits_{j=1}^{n_0} \left\|
                        \left[\mathstrut m_j\left(\v^1,\ldots,\v^{n_1}\right) - \hat{x}_j\right]\e^1 + 
                        \sum\limits_{i=1}^{n_1} M_{ij}\left(\v^1,\ldots,\v^{n_1}\right) \v^i \right\|_2^2 \leq \rho^2,
        \end{array}
    $$
    which is also represented by \eqref{eq:formulation_with_z_e1} with appropriate $m_j$ and $M_{ij}$.
    By Lemma~\ref{lem:collinearity_of_z_on_formulation_with_e1},
        if $((\v^1)^*,\ldots,$ $ (\v^{n_1})^*)$ is an optimal solution of the above problem,
        all $n_1$ points  $(\v^1)^*,\ldots,(\v^{n_1})^*$ can be  collinear with $\e^1$.
    Since each $\left(\u^j\right)^*$ depends on $(\v^1)^*,\ldots,(\v^{n_1})^*$,
        by applying Lemma~\ref{lem:solution_form_of_second_stage_ellipsoid} again,
        $\left(\u^1\right)^*, \ldots, \left(\u^{n_0}\right)^*$ are also collinear with $\e^1$ 
        under the points $(\v^1)^*,\ldots,(\v^{n_1})^*$.
    Therefore, Lemma~\ref{lem:rank_one_solution_collinear} guarantees the existence of a rank-1 matrix solution 
    to \eqref{eq:primal_for_deepsdp_onelayer}.
    The second result follows from strong duality in Lemma~\ref{lem:strong_duality_assumption} and \cref{fig:correspondence_tightness}.
\end{proof}

\subsection{Rectangular input} \label{ssec:box_input_onelayer_deepsdp}

We consider the case that the input set $\XC$ is a rectangle, {\it i.e.,}
\begin{equation} \label{eq:input_region_box}
    \XC \coloneqq \left\{ \x \,\middle|\, \left| x_j - \hat{x}_j \right| \leq \rho_j \; \text{for all $j \in \{1,\ldots,n_0\}$} \right\},
\end{equation}
where $\hat{\x} \in \Real^{n_0}$ is the center of the rectangle,
and $\rho_j$ is the length of the rectangle along with the $j$th dimension.
For \eqref{eq:input_region_box}, the matrix set $\PC_\XC$ becomes
$$
    \PC_\XC = \left\{
        \sum_{j=1}^{n_0} \gamma_j \begin{bmatrix}
            \hat{x}_j^2 - \rho_j^2 & -\hat{x}_j\trans{\left(\e^j\right)} \\
            -\hat{x}_j\e^j & \e^j\trans{\left(\e^j\right)}
        \end{bmatrix}
        \,\middle|\, \boldsymbol{\gamma} \geq \0,\; \boldsymbol{\gamma} \in \Real^{n_0} \right\},
$$
where $\boldsymbol{\gamma} \in \Real^{n_0}$ is the dual variable of DeepSDP~\eqref{eq:introduction_deepsdp}.
The corresponding constraints in the primal SDP relaxation are
$$
    \ip{M_\mathrm{in}\left(\begin{bmatrix}
            \hat{x}_j^2 - \rho_j^2 & -\hat{x}_j\trans{\left(\e^j\right)} \\
            -\hat{x}_j\e^j & \e^j\trans{\left(\e^j\right)}
        \end{bmatrix}\right)}{G} \leq 0, \qquad j = 1,\ldots,n_0.
$$
We consider the same substitution of $G$ as \eqref{eq:primal_for_deepsdp_onelayer} with $\u^1,\dots,\u^{n_0}$
and $\v^1,\dots,\v^{n_1}$.
For any $j \in \{1,\ldots,n_0\}$, the left-hand side of the above inequality is
\begin{align*}
    \ip{M_\mathrm{in}\left(\begin{bmatrix}
                \hat{x}_j^2 - \rho_j^2 & -\hat{x}_j\trans{\left(\e^j\right)} \\
                -\hat{x}_j\e^j & \e^j\trans{\left(\e^j\right)}
            \end{bmatrix}\right)}{G}
    &= \trans{\left(\u^j\right)}\u^j - 2\hat{x}_j \trans{\e}\u^j + \hat{x}_j^2 \trans{\e}\e - \rho_j^2 \\
    &= \left\| \u^j - \hat{x}_j \e \right\|_2^2 - \rho_j^2.
\end{align*}
Thus, the problem~\eqref{eq:nc_interpretation_for_onelayer} can be transformed into
\begin{align} \label{eq:nc_interpretation_for_onelayer_box}
        \min\limits_{\u^j,\v^i} \ \ \eqref{eq:nc_interpretation_for_onelayer_objective} \quad
        \subto \ \
          & \eqref{eq:nc_interpretation_for_onelayer_nonnegative}, \,
            \eqref{eq:nc_interpretation_for_onelayer_value}, \,
            \eqref{eq:nc_interpretation_for_onelayer_eithercondition}, \nonumber \\
          & \|\u^j - \hat{x}_j\e\|_2 \leq \rho_j^2, \quad j = 1,\ldots,n_0.
\end{align}
To derive a two-stage problem as in \cref{sec:oneneuron},
    we assume the following.
\begin{assum} \label{assum:identity_w0}
    $W^0$ is the identity matrix, {\it i.e.,}  $n_0 = n_1$ and $W^0 = I_{n_0}$.
\end{assum}
\noindent
Then, \eqref{eq:nc_interpretation_for_onelayer_value} and \eqref{eq:nc_interpretation_for_onelayer_eithercondition} are reduced to
\begin{align*}
    \trans{\e}\v^j &\geq {\textstyle \trans{\e}\left( \u^j + b^0_i\e \right)}, \quad & & j = 1,\ldots,n_0, \\
    \left\|\v^j\right\|_2^2 &\leq {\textstyle \trans{\left( \u^{j} + b^0_j\e \right)}\v^j}, \quad & & j = 1,\ldots,n_0,
\end{align*}
and the following two-stage problem can be obtained:
\begin{equation} \label{eq:first_stage_for_onelayer_box}
    \begin{array}{rl}
        \min\limits_{\v^1,\ldots,\v^{n}} & \eqref{eq:nc_interpretation_for_onelayer_objective} \\
        \subto & \eqref{eq:nc_interpretation_for_onelayer_nonnegative}, \quad \Psi^j_\mathrm{rect}(\v^j) \leq \rho_j^2, \; j = 1,\ldots,n_0,
    \end{array}
\end{equation}
and
\begin{equation} \label{eq:second_stage_for_onelayer_box}
    \begin{array}{rl}
        \Psi^j_\mathrm{rect}(\z) \coloneqq
        \min\limits_{\u^j} & \|\u^j - \hat{x}_j\e\|_2 \\
        \subto
            & \trans{\e}\z \geq {\textstyle \trans{\e}\left( \u^j + b^0_i\e \right)}, \\
            & \left\|\z \right\|_2^2 \leq {\textstyle \trans{\left( \u^{j} + b^0_i\e \right)}\z}.
    \end{array}
\end{equation}
When $\e = \e^1$, the second-stage problem $\Psi^j_\mathrm{rect}(\z)$ is equivalent to $\Phi(\z)$ defined in \eqref{eq:oneneuron_second_stage}.
Thus, $\Psi^j_\mathrm{rect}(\z) =\|\z - \hat{x}_j\e^1\|_2^2$ follows from \cref{prop:value_second_stage},
    and the tightness of DeepSDP~\eqref{eq:introduction_deepsdp} can be shown as follows.
\begin{theorem}
    Let $\e$ be an arbitrary unit vector with $\|\e\| = 1$.
    Suppose that Assumptions~\ref{asm:polytope_safety_set}, \ref{asm:no_effect_lastlayer}, and \ref{assum:identity_w0} hold.
    Then, any solution $\left\{\left(\u^1\right)^*,\ldots,\left(\u^n\right)^*,\left(\v^1\right)^*,\ldots,\left(\v^n\right)^*\right\}$ of \eqref{eq:nc_interpretation_for_onelayer_box} 
        are collinear with $\e$.
    Thus, the problem~\eqref{eq:primal_for_deepsdp_onelayer} has a rank-1 matrix solution.
    In addition, if both \eqref{eq:primal_for_deepsdp_onelayer} and \eqref{eq:introduction_deepsdp} 
    have optimal solutions,
    and the feasible set of \eqref{eq:introduction_deepsdp} is bounded,
    then DeepSDP \eqref{eq:introduction_deepsdp} is a tight relaxation for the original QCQP~\eqref{eq:qcqp_for_deepsdp}.
\end{theorem}
\begin{proof}
    Without loss of generality, we may assume that $\e = \e^1$.
    By Proposition~\ref{prop:value_second_stage},
        the optimal solution of the second-stage problem \eqref{eq:second_stage_for_onelayer_box} is $\u^j = \v^j - b^0_i\e$,
        and the first-stage problem \eqref{eq:first_stage_for_onelayer_box} is
    \begin{equation*}
        \begin{array}{rl}
            \min\limits_{\v^i} &  {\textstyle 2\sum_{i=1}^{n_1} c_i \, \trans{\left(\e^1\right)}\v^i} \\
            \subto
            & \trans{\left(\e^1\right)}\v^i \geq 0, \; i = 1,\ldots,n,  \\
            & \left\|\v^i - \left(b^0_i + \hat{x}_i\right)\e^1\right\|_2^2 \leq \rho_i^2, \; i = 1,\ldots,n.
        \end{array}
    \end{equation*}
    Then, the same discussion in Theorem~\ref{thm:tightness_oneneuron} can be applied to each $\v^i$,
        therefore, the first and second results of this theorem follows by Lemmas~\ref{lem:rank_one_solution_collinear} and 
        \ref{lem:strong_duality_assumption}, respectively.
\end{proof}

\section{Conclusion}\label{sec:conclusion}


We have presented sufficient conditions under which DeepSDPs for single-layer NNs are tight in three cases.
A common aspect of these cases is the identification of a polytope safety specification set.
For the first case, we have proved that
    the DeepSDP for a single-neuron NN provides a tight optimal solution  if the given vectors $\hat{x}$ and $b^0$ satisfies $\hat{x} \geq -b^0$.
In the second case for the DeepSDP with ellipsoidal inputs,
   we have proved that the DeepSDP is always a tight relaxation without any assumptions.
The Karush-Kuhn-Tucker (KKT) conditions have been used  to analyze the optimal solutions $(\u^j)^*, (\v^i)^*$ of the first- and second-stage problems,
    as shown in Lemma~\ref{lem:solution_form_of_second_stage_ellipsoid}.
For the third case where the input set is a rectangle, we have shown that
    the auxiliary two-stage problem can be reduced into the one in the first case
    and the tightness can be determined by the condition for the single-neuron NNs.

For  future work, it would be interesting to investigate  the extent to  which the tightness condition can be weakened  
    by incorporating the local quadratic constraints proposed in \cite{Fazlyab2022}, 
    which have not been included in this work.
The DeepSDPs that satisfy the tightness conditions of this work remain tight even with the inclusion of local constraints.
As a result, incorporating the local constraints may lead to tighter SDP relaxations even under milder assumptions.
Also, examining the tightness of SDP relaxations for QCQPs with  the ReLU function as a quadratic constraint, which was
 developed for different purposes,
 would be an interesting direction.
 Our approach, which  decomposes the problem into a two-stage problem and use the KKT conditions, may be useful for solving the QCQPs.

\vspace{0.5cm}


\subsection*{Declarations} 

{\bf Conflict of interest} The authors have no conflict of interest to disclose.

\noindent

\bibliographystyle{abbrv} 


\end{document}